\documentclass[11pt]{article}
\usepackage{amsfonts}
\usepackage{amssymb,amsmath}
\usepackage{stmaryrd}
\usepackage[utf8]{inputenc}
\usepackage[english]{babel}
\usepackage{amsthm}
\usepackage[a4paper, total={6in, 8in}]{geometry}
\usepackage{color}
\usepackage{xcolor}
\usepackage{graphicx}
\usepackage{hyperref}
\usepackage{ulem}

\usepackage{dsfont}

\newcommand{\ve}{\varepsilon}
\newcommand{\drift}{\rho}

\def\be{\begin{eqnarray}}
	\def\ee{\end{eqnarray}}
\def\ben{\begin{eqnarray*}}
	\def\een{\end{eqnarray*}}

\def\me{\medskip\noindent}

\newcommand{\Co}{C}

\newcommand{\card}{\mbox{Card}}

\newcommand{\supp}{\mbox{supp}}

\def\D{\mathbb{D}}

\def\N{\mathbb{N}}

\def\R{\mathbb{R}}

\def\E{\mathbb{E}}

\def\ind{{\mathchoice {\rm 1\mskip-4mu l} {\rm 1\mskip-4mu l}
		{\rm 1\mskip-4.5mu l} {\rm 1\mskip-5mu l}}}

\newtheorem{thm}{Theorem}[section]
\newtheorem{lem}[thm]{Lemma}
\newtheorem{corollary}[thm]{Corollary}
\newtheorem{ex}[thm]{Example}

\newtheorem{prop}[thm]{Proposition}
\newtheorem{rem}[thm]{Remark}

\renewenvironment{proof}{\noindent {\bf Proof \phantom{9}}}
{\hfill $\square$ \vspace{0.25cm}}

\title{Time reversal of spinal processes for linear and non-linear branching processes near stationarity}

\author{Beno\^{i}t Henry\thanks{IMT Lille Douai, Institut Mines-T\'el\'ecom, Univ.\ Lille, F-59000 Lille, France; E-mail: \texttt{benoit.henry@imt-lille-douai.fr}}, Sylvie
	M\'el\'eard\thanks{Institut Universitaire de France and Ecole polytechnique, CNRS, Institut polytechnique de Paris, CMAP, route de
		Saclay, 91128 Palaiseau Cedex-France; E-mail: \texttt{sylvie.meleard@polytechnique.edu}}, Viet Chi Tran\thanks{LAMA, Univ Gustave Eiffel, Univ Paris Est Creteil, CNRS, F-77454 Marne-la-Vall\'ee, France; E-mail:
		\texttt{chi.tran@univ-eiffel.fr}}}

\date{\today}

\begin{document}
	
	\maketitle
	
	\begin{abstract}
				We consider a stochastic individual-based population model with competition, trait-structure affecting reproduction and survival, and changing environment. The changes of traits are described by jump processes, and the dynamics can be approximated in large population by a non-linear PDE with a non-local mutation operator. Using the fact that this PDE admits a non-trivial stationary solution, we can approximate the non-linear stochastic population process by a linear birth-death process where the interactions are frozen, as long as the population remains close to this equilibrium. This allows us to derive, when the population is large, the equation satisfied by the ancestral lineage of an individual uniformly sampled at a fixed time $T$, which is the path constituted of the traits of the ancestors of this individual in past times $t\leq T$. This process is a time inhomogeneous Markov process, but we show that the time reversal of this process possesses a very simple structure (e.g. time-homogeneous and independent of $T$). This extends recent results where the authors studied a similar model with a Laplacian operator but where the methods essentially relied on the Gaussian nature of the mutations.
	\end{abstract}
	
	\me Keywords: stochastic individual-based models, birth-death processes, interaction, competition, jump process, non-local mutation operator, many-to-one formulas, ancestral path, genealogy, phylogeny.

	\bigskip
	\noindent \emph{MSC 2000 subject classification:} 92D25, 92D15, 60J80, 60K35, 60F99.
	\bigskip
	
	\noindent \textit{Acknowledgements}: This work has been supported by the Chair ``Mod\'elisation Math\'ematique et Biodiversit\'e'' of Veolia Environnement-\'Ecole Polytechnique-Museum National d'Histoire Naturelle-Fondation X. V.C.T. also acknowledges support from Labex B\'ezout (ANR-10-LABX-58).\\
	
	\section{Introduction}
	
	We are interested in describing the ancestry of an individual sampled from a trait-structured population whose dynamics is ruled by births, deaths, mutations and environmental changes. More precisely, we consider as a toy model a stochastic individual-based population model in continuous time, with variable size, and in which each individual is characterized by its own trait $x$ which is interpreted here as its fitness. For simplicity, the trait $x$ is considered to be real-valued. This trait can change through time (by mutations occurring continuously in time). The case where it is driven by a Brownian motion has been considered in a previous paper by the authors \cite{CHMT}. Here, we are interested in a non-local mutation kernel. Computations exploiting the Gaussian nature of the mutations can not be used any more. We base our work on duality properties satisfied by the semi-groups and generators underlying the mutations and environmental changes.\\
	The interest in ancestries and phylogenies (the trait values of ancestors of the population) has developed in recent years as the phylogenies provide a new understanding for the evolution of the biodiversity in response to the ecological dynamics or environmental changes (e.g. \cite{neherbedford}).\\

	We will be interested in large population limits and the model is parameterized by an integer $K$ (think of the carrying capacity for instance) that we will let go to infinity. The size of the population is then $N^K_t$ at time $t>0$. An individual of trait $x\in \R$ gives birth to a new individual of same trait at rate $b(x)$ and dies at the rate $d(x) + N_t^K/K$. In the death rate, the term $d(x)$ corresponds to the natural death to which is added a competition term expressing the additional death rate exerted by the interaction with the other individuals in the population. Here, this competition is assumed of logistic type, i.e. it is proportional to the size $N_t^K$ and does not account for the whole trait distribution. During their life, the trait of an individual mutates according to a kernel $\gamma m(x,y)dy$ and experiences a linear drift with environmental velocity $\drift \in \R$ due to environmental changes (see \cite{CHMT} for details). We assume that $\gamma>0$ is the jump rate and that $m(x,y)dy$ is the probability measure describing the jumps (assumed to be absolutely continuous with respect to the Lebesgue measure, for the sake of simplicity).
	
	Individual labels can be chosen in the Ulam-Harris-Neveu set $ \mathcal{I} = \cup_{n\in \N}\mathbb{N}^n$ (e.g. see \cite{legall}) where offspring labels are obtained by concatenating the label of their parent with their ranks among their siblings. This set is endowed with a partial order $\prec$, where $i\prec j$ if there exists $i'\in \mathcal{I}$ so that $j$ is the concatenation of the chains of integers $i$ and $i'$. We denote by $V^K_t\subset \mathcal{I}$ the set of labels of individuals alive at time $t$ (implying that $N^K_t=\card(V^K_t)$) and by $X^i_t$ the trait of the $i$-th individual at this time. The lineage of the individual $i\in V^K_t$ consists in the path defined from $[0,t]$ to $\R$ and that associates to $s$ the trait of the closest ancestor to $i$ living at time $s$, and that we will denote by $X^i_s$. Such path is càdlàg because of the mutation kernel and can be extended to a function of the Skorokhod space $\D=\mathbb{D}(\R_+,\R)$	by setting it to the constant value equal to $X^i_t$ for times $s>t$. Also, we will say that this path $(X^i_t, t>0)$ is `forward in time', in opposition to the ancestral path $(X^i_{T-t}, t\in [0,T])$ of an individual $i\in V^K_T$ for a given $T>0$ that is considered in `backward in time'.
	The set of all lineages for living individuals at time $t$ can be represented by the following point measure on $\D$:
	\begin{equation}H^{K}_{t} = \frac{1}{K} \sum_{i\in V^K_t} \delta_{(X^i_{s\wedge t},\ s\in \R_+)}.\label{def:HK}
	\end{equation}Denoting by $\mathcal{M}_f(\D)$ the set of finite measures on $\D(\R_+,\R)$, the process $(H^K_t)_{t>0}$ is a càdlàg process of $\D(\R_+,\mathcal{M}_f(\D))$ which is the \textit{historical particle system}, following the terminology and concept introduced by Dawson Perkins \cite{dawsonperkinsAMS,perkinsAMS}, Dynkin \cite{dynkin91} (see also \cite{dawsonhochbergvinogradov,grevenlimicwinter}).
	The spaces $\D$ and $\D(\R_+,\mathcal{M}_f(\D))$ are equipped with the Skorokhod topology and $\mathcal{M}_{f}(\D)$ is equipped with the topology of weak convergence (see e.g. \cite{billingsley2013convergence}).
	Méléard and Tran \cite{meleardtran_suphist} and Kliem \cite{kliem} have studied limits of this process under a diffusive scaling when $K\rightarrow +\infty$. In a recent work  \cite{CHMT}, we have studied a similar historical process in large population, {without rescaling of time} and with particles undergoing Brownian motion. We obtained  the distribution, backward in time, of a typical ancestral lineage (the lineage of an individual $i\in V^K_T$ sampled uniformly among the population living at time $T$),  using extensively explicit computation based on Brownian properties.  In the present paper, we extend these results to the case where the motion is a drifted jump process  with generator:
	\begin{equation}\label{def:L}
	L\varphi(x)=\drift\ \partial_{x}\varphi(x)+\gamma \int_{\mathbb{R}}(\varphi(y)-\varphi(x))\ m(x,y)dy,
	\end{equation} where $\gamma$, $\rho$ and $m(x,y)$ have been introduced above. The jump part corresponds to mutations and the drift part corresponds to the environmental changes. Given a collection of independent  Poisson point measures $(Q^i(ds,dy,d\theta), i\in \mathcal{I})$ on $\R_+\times \R \times \R_+$ with common intensity measure the Lebesgue measure, the trait dynamics  $X^i$ of individual $i$  solves:
	\begin{equation}\label{eq:EDS_intro}
	    X^i_t=X^i_0+ \rho t+\sum_{j\in \mathcal{I}} \int_0^t \int_{\R} \int_{\R_+} \ind_{\{j=i(s),\ \theta\leq \gamma m(X^i_{s_-},y)\}}\big(y-X^i_{s_-}\big) Q^j(ds,dy,d\theta).
	\end{equation}where $i(s)=\max_\prec\{j\in V^K_s,\ j\prec i\}$ denotes the index of the most recent ancestor of $i$ living at time $s$.

	The process $(Z^K_t)_{t\in \R_+}$ corresponding to the trait distribution of the living individuals at time $t>0$,
	\begin{equation}\label{def:ZK}
	    Z^K_t(dx)=\frac{1}{K}\sum_{i\in V^K_t} \delta_{X^i_t}(dx),
	\end{equation}converges in the limit $K\rightarrow +\infty$ in $\mathbb{D}(\R_+,\mathcal{M}_f(\R))$ to the solution of the following partial differential equation (PDE):
		\begin{equation}\label{eq:pdeintro}
		\partial_{t}f_{t}(x)=-\drift\ \partial_{x}f_{t}(x)+\gamma \int_{\mathbb{R}}(f_{t}(y)-f_{t}(x))\ m(y,x)dy +  \left(h(x)-\int_{\mathbb{R}}f_{t}(y)dy  \right)f_{t}(x),
	\end{equation}
	where $$h(x)=b(x)-d(x)$$ is the natural growth rate.
	In \cite{CHMT}, the surprising result is that the random ancestral lineage, when reversed in time, becomes a simple Ornstein-Uhlenbeck process whose laws is time homogeneous and independent of $T$. This phenomenon is unusual in the setting of spinal processes theory since, in general, processes both dependent in time and $T$ arise. Unfortunately, the method developed in \cite{CHMT} essentially relies on the Brownian nature of the particles' motions and the quadratic term in the death rate as this allows many explicit computations. This raises the question of whether the simplicity of the time-reversed spinal process is due to the particular context of \cite{CHMT}. \\
	The results of the present paper extend the ones of \cite{CHMT} in several directions as we do not require the particles' motions to be of Brownian type and the death-rate to include a quadratic term. More importantly, the method developed in the article is far more robust to other extensions (for instance replacing the jump operator by a jump-diffusion operator).
	The generator of the ancestral path of the randomly chosen individual (in forward time) can be obtained by following the work of Marguet \cite{Marguet2}  and other earlier works (see also \cite{CHMT,cloez2017limit} and notably \cite{harris1996large,hardy2006new} for the Feynmann-Kac formula for branching processes): this path is a Markov process inhomogeneous in time. The time-reversal of this process is obtained by following techniques developped by Chung and Walsh, Nagasawa, Reinhard and Roynette \cite{nagasawa64,nagasawa,chung1969reverse,walsh1970time,reinhard} (see also \cite{DM2}) and we will see that it is a homogeneous Markov process. These techniques are based on a duality theory for semigroups which are particularly well-suited in our context as many-to-one formulas express an intrinsic duality structure within branching processes.




	Informally speaking, we prove the following theorem which characterizes the law of the time reversed spinal process. This result is made more precise in Theorem \ref{thm:conclusion}.
	\begin{thm}\label{reversed}
		Assuming that the initial trait distribution $Z^K_0$ of the population converges to the stationary solution $F(x)dx$ of \eqref{eq:pdeintro}, the process describing, backward in time, the lineage of an individual sampled in the living population at time $T>0$ converges, when $K\rightarrow +\infty$, to a time homogeneous Markov process $Y$ whose law is independent of $T$ and characterized by its semigroup $(P^{R}_{t})_{t}$ acting on any bounded measurable function $\varphi$:
		\begin{equation}\label{eq:PR_intro}
		P^{R}_{t}\varphi=\frac{1}{F}\widehat{P^{\ast}_{t}}(F\varphi)
		\end{equation}
		where $\widehat{P^\ast_t}$ is defined by
		\[
		\widehat{P^{\ast}_{t}}\varphi(x)=\mathbb{E}_{x}\left[\exp\left(\int_{0}^{t}(h(X^{\ast}_{s})-\lambda)\ ds\right)\varphi(X^{\ast}_{t})\right],
		\]with $X^{\ast}$ a Markov process whose generator is the formal adjoint $L^\ast$ of $L$ and $\lambda=\int_\mathbb{R}F(x)\ dx$.
	\end{thm}
	In terms of generator, this says that the time reversed process ${Y}^R$ of the spinal process $Y$ has the infinitesimal generator (see Proposition \ref{prop:generatorReversed}) given  by
	\begin{align}\label{eq:LRintro}
	L^R\varphi(x)= & \frac{L^*(F\varphi)(x)}{F(x)}+(h(x)-\lambda)\varphi(x)
	=  \rho \varphi'+\gamma \int_{\mathbb{R}}\left(\varphi(y)-\varphi(x)\right)\ \frac{F(y)}{F(x)}m(y,x)dy 
	\end{align}
	whenever this makes sense. We can see, as in the Gaussian case developed in \cite{CHMT}, that the ancestral lineage of a typical individual backward in time has a very simple dynamics: here, the jump measure is biased according to the stationary distribution $F$. {Notice that the expressions \eqref{eq:PR_intro} and \eqref{eq:LRintro} also hold in the Gaussian case. In fact, these expressions are quite general and could be generalized to mutation mechanisms other than the Gaussian setting of \cite{CHMT} or the case considered here, provided we can prove that the PDE associated with the large population approximation of the trait distribution (see here \eqref{eq:pdeintro}) admits a unique stationary measure $F$.}\\
	
	\begin{figure}[!ht]
	\centering
	\begin{tabular}{cc}
	\includegraphics[height=6cm, width=6cm]{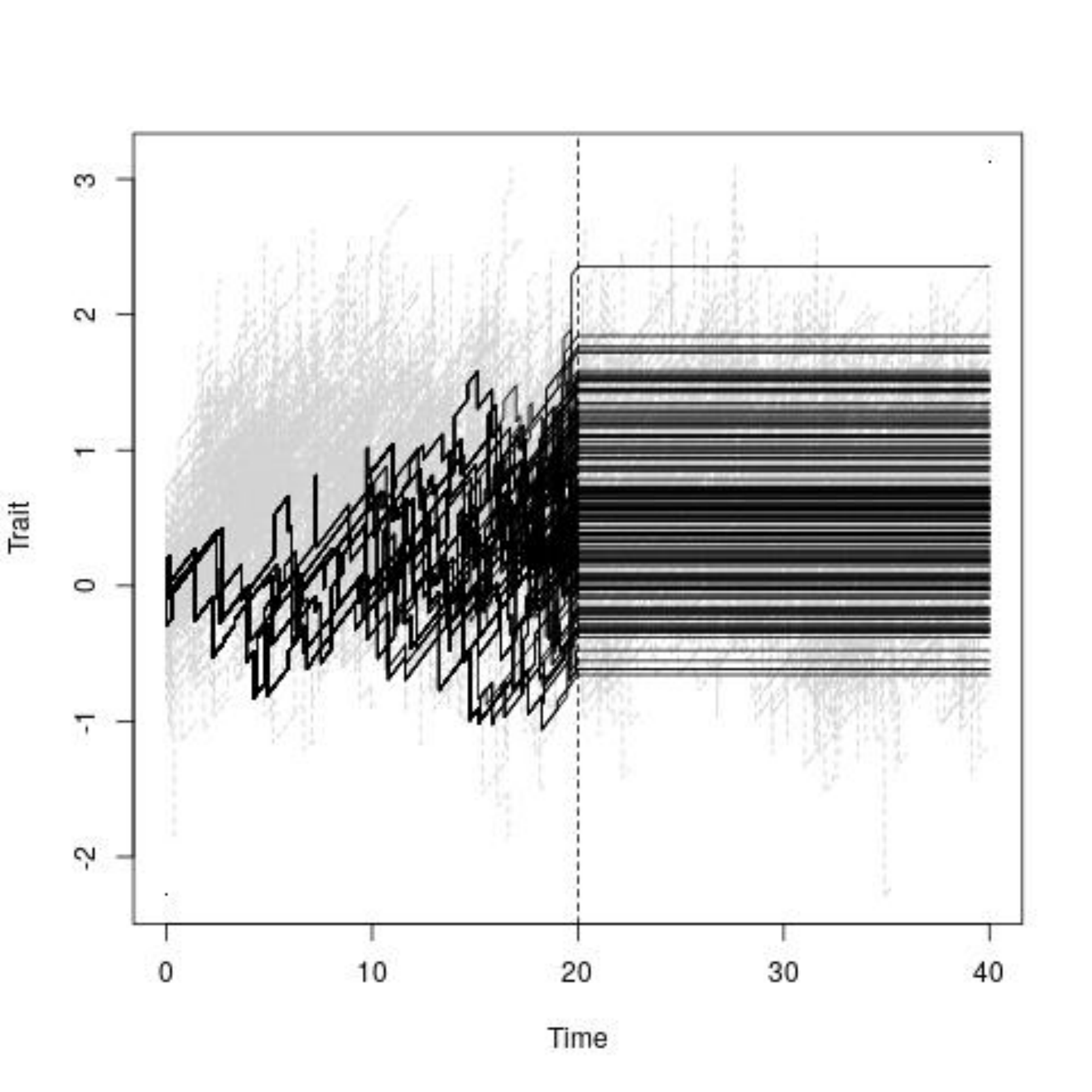} & \includegraphics[height=6cm, width=6cm]{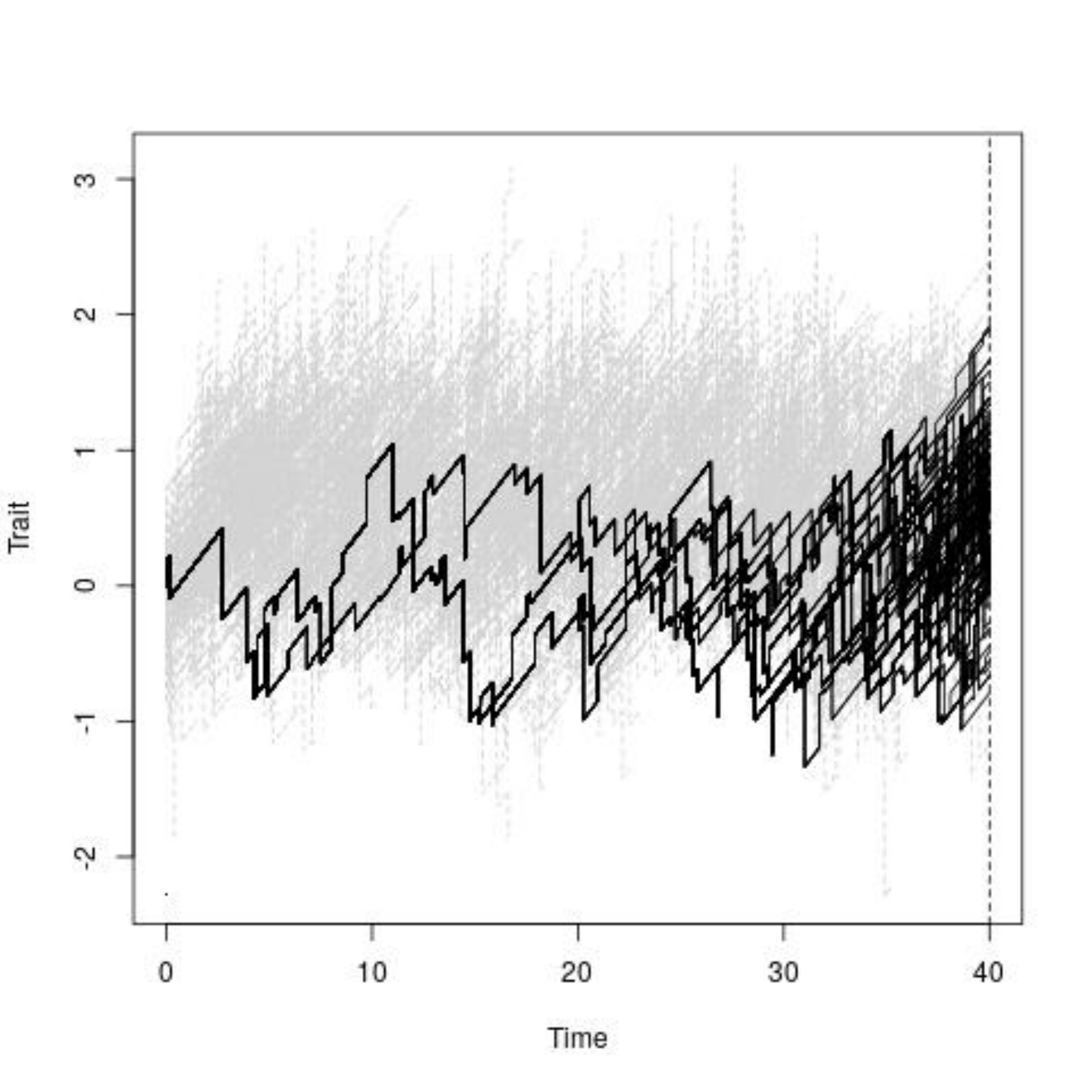}\\
	(a) & (b)
	\end{tabular}
	\caption{{\small \textit{Ancestral lineages of the living individuals at two different times, in black. The sampling time is represented by the vertical line. In gray, all the individuals that have lived are pictured, which allow to show the traits occupied by past lost lineages.}}}\label{fig:simu}
	\end{figure}
	In a recent work \cite{patout_ancestral_2020}, the semigroup $P^R$ is introduced and justified from a macroscopic point of view using the so-called neutral fractions that have been introduced by \cite{roquesgarnierhamelklein} to keep track of ancestries in PDEs that describe macroscopic populations. The present work provides a rigorous mathematical justification of these semi-groups grounded on an individual-based model and a time inversion of the typical ancestral line.\\

	The article is organized as follows. Section \ref{sec:model} aims to describe the setting of this work: Section \ref{ssec:duality} introduces the process describing the motion of the trait of the individuals and its dual process which plays an important role in the following. Section \ref{ssec:hist} provides the deterministic equations approximating the dynamics of $Z^K$ and $H^K$, and whose stationary solutions are studied in Section \ref{ssec:stationary}.  These stationary  solutions are central in our approach. Indeed,  using these solutions as limiting initial conditions for the historical process allows its approximation by a linear branching process. The coupling of the processes $Z^K$ and $H^K$ with linear births and deaths processes $\widetilde{Z}^K$ and $\widetilde{H}^K$ is presented in Section \ref{section:coupling}. For the processes $\widetilde{H}^K$, the branching property holds and we can use many-to-one formulas to obtain asymptotic representations of the ancestral lineages.  In Section \ref{section:linear}, the spine of $\widetilde{H}^K$, i.e. the ancestral lineage of a typical individual chosen at time $T$, is studied and in particular its time reversal (Section \ref{sec:Yreturn}). We then conclude in Section \ref{sec:return} and \eqref{eq:LRintro} is established.\\

	
\noindent \textbf{Notations:}	In the sequel, we will denote by $\mathbb{L}^\infty=\mathbb{L}^\infty(\R)$ the set of measurable bounded functions on $\R$ and by $\mathbb{L}^1=\mathbb{L}^1(\R)$ the set of functions that are integrable with respect to the Lebesgue measure on $\R$.
	$C_b=C_b(\R)\subset \mathbb{L}^\infty(\R)$ is the set of bounded continuous functions and $C^1_b=C^1_b(\R)$ the subset of bounded differentiable functions such that the derivative $f'\in C_b$.
	From now, for any two measurable functions $f$ and $g$, $\langle f, g \rangle$ stands for $\int_{\mathbb{R}}f(x)g(x)\ dx$ whenever this last expression makes sense. Similarly, for a finite measure $\mu$ and a measurable function $f$, $\langle \mu,f\rangle=\int_\R f(x) \mu(dx)$ whenever this integral is well defined.
	
\section{Models and settings}\label{sec:model}


	\subsection{Individual based model and hypotheses}
	\label{ssec:duality}

 Recall that the population at time $t$ can be represented by the point measure $Z^K_t$ defined in \eqref{def:ZK} and that the ancestries of the living individuals at time $t$ are given by $H^K_t$ defined in \eqref{def:HK}.  The trait of an individual evolves during its life according to the drifted jump process with generator $L$ defined in \eqref{def:L}. We will denote by $(X_t)_{t\in \R_+}$ the Markov process with infinitesimal generator $(L,D(L))$, with $C_{b}^1 \subset D(L)$.


Before going further let us precise the hypotheses that we need and that will be assumed satisfied throughout this work. \\
	
{\bf Assumptions $(H)$}

	\begin{itemize}
		\item[(a)] $(x,A)\in\mathbb{R}\times \mathcal{B}(R)\to \int_A m(x,y)dy$ is weakly continuous in the first variable and satisfies
		$$\exists \varepsilon>0,\kappa_{0}>0, \forall x\in\mathbb{R},\ m(x,y)\geq \kappa_{0}\mathds{1}_{(x-\varepsilon,x+\varepsilon)}(y) $$
		\item[(b)] $h$ is continuous, $h(0)>0$ and there exists $c\in\mathbb{R}$ such that  $\ \forall x\in\mathbb{R}$,  $h(x)\leq c$, and $\lim\limits_{x\to\pm\infty}h(x)=-\infty$.
		\item[(c)] There exist $q\geq 1$ and $x_0>0$ such that for all $|x|\geq x_0$, $h(x)\leq -|x|^q$, and
	    \begin{equation}\label{hyp:moment-q}\sup_{x\in\mathbb{R}}\int_{\mathbb{R}}y^{2q} m(x,y)dy<+\infty.\end{equation}
		\item[(d)] $\forall y\in\mathbb{R},\ \int_{\mathbb{R}}m(x,y)\ dx=1= \int_{\mathbb{R}}m(y,x)\ dx.$
			\end{itemize}

	
Let us comment on these assumptions. Assumptions $(H.a)-(H.b)$ are meant to provide the existence {and uniqueness} of a non-trivial stationary solution to Equation \eqref{eq:pdeintro}. These are borrowed from \cite{CG} where Cloez and Gabriel studied a related eigen-problem. The first part of Assumption $(H.a)$ and  Assumption $(H.c)$  are made  to prove the convergence of the particle systems $Z^K$ to the limiting PDE \eqref{eq:pdeintro} (see the hypothesis of Theorem \ref{thm-CV}). Assumption $(H.c)$ is not so restrictive, as shown in the following examples, and could be easily changed for other growth rates $h$ provided that $(H.b)$ holds. Whether $(H.a)$ and $(H.b)$ can be further weakened is a difficult problem (see \cite{CG,Alfaro}). Assumption $(H.d)$ allows us to give a simple and straightforward definition of the dual process of $X$. Our work can be extended to the case where $\int_\R m(x,y)dx<+\infty$ provided we add the correct renormalizations. These hypotheses can certainly be weakened as the adjoint process always exists \cite{DM2}, but we choose to use these assumptions as a trade-off between simplicity and generality.
	
	\begin{ex}
1. In \cite{CHMT}, the following birth and death rates are used $b(x)=1$, $d(x)=x^2/2$, so that $h(x)=1-x^2/2$ satisfies $(H.b)$ and $(H.c)$. The assumption $(H.c)$ is satisfied provided
	\[\sup_{x\in\mathbb{R}}\int_{\mathbb{R}}y^4 m(x,y)dy<+\infty.\]
	2. The case of a convolution operator, \textit{i.e.} where $m(x,y)=\widetilde{m}(x-y)$ for some continuous probability density $\widetilde{m}$ satisfying $\widetilde{m}(0)>0$ and \eqref{hyp:moment-q}, also enters the Assumptions (H).\\
	For instance, a Gaussian mutation kernel and a polynomial growth rate $h$ satisfy these Assumptions.
	\end{ex}

	\subsection{Limiting PDE}
	\label{ssec:hist}
	Let $T>0$. The processes $(Z^K_t)_{t\in [0,T]}$  are solutions of stochastic differential equations  driven by Poisson point measures (see Appendix \ref{app:SDE-ZK-HK}). In this section, we  study their asymptotic behavior   when $K\rightarrow +\infty$, assuming that the initial conditions $Z^K_0$ converge to a non-trivial measure $\xi_0$,  assumed to be deterministic for  sake of simplicity.
 The dynamics of measures is described with respect to test functions: for finite measures on $\R$, such as the $Z^K_t$'s, we   consider $\varphi\in C^1_b(\R)$ which is a dense space of $C_b(\R)$ for the uniform norm.

	\begin{thm}\label{thm-CV}
Assume that the hypotheses $(H)$ are satisfied. Let us assume that the initial conditions $(Z^K_0(dx))_{K}$ satisfy \begin{equation}\label{hyp:moment}
			\sup_{K\in \N^*}\E\big(\langle Z^K_0,1\rangle^{2+\epsilon} \big)<+\infty\qquad \mbox{ and }\qquad \sup_{K\in \N^*} \E\big( \langle Z^K_0,x^{2q}\rangle^{1+\epsilon}\big)<+\infty.
		\end{equation}and that the sequence $(Z^{K}_0(dx))_{K}$ converges in probability (and weakly as measures) to the deterministic finite measure $\,\xi_0(dx)$. Let $T>0$ be given. Then the sequence of processes $\,(Z^{K}_{t})_{ t\in [0, T]}$  converges in  $\mathbb{L}^2$, in $\D([0,T],\mathcal{M}_{f}(\R))$ to a deterministic continuous function $\,( \xi_{t})_{t\in [0, T]}$ of $\Co([0,T],\mathcal{M}_{f}(\R))$ that is the unique solution of the weak equation: $\forall \varphi\in \Co^1_{b}(\R)$,
		\begin{equation}
			\label{limit-moving}\langle \xi_{t},\varphi\rangle=\langle \xi_{0},\varphi\rangle+\int_{0}^{t} \int_{\mathbb{R}}  \left\{ \left(h(x)-\langle \xi_{s}, 1\rangle\right)\varphi(x)+L\varphi(x)\right\} \xi_{s}(dx)\, ds.\end{equation}
		More precisely:\begin{equation}
			\label{CV-L2}\lim_{K\to \infty}\E(\sup_{t\le T} |\langle Z^{K}_{t},\varphi\rangle - \langle \xi_{t},\varphi\rangle|^2) = 0.
		\end{equation}Moreover, we have  that $\ \sup_{t\in [0,T]}\langle \xi_{t}, 1+x^2\rangle <+\infty$.
		\end{thm}

\bigskip
Idea of the proof: Under the assumptions \eqref{hyp:moment}, we can adapt the proof of Lemma B.1 in \cite{CHMT} to obtain that
	 \begin{equation}\label{eq:moment}
			\sup_{K\in \N^*}\E\big(\sup_{t\in [0,T]}\langle Z^K_t,1\rangle^{2+\epsilon} \big)<+\infty\qquad \mbox{ and }\qquad \sup_{K\in \N^*} \E\big( \sup_{t\in [0,T]}\langle Z^K_t,x^{2q}\rangle^{1+\epsilon/2}\big)<+\infty.
		\end{equation}Then, Theorem \ref{thm-CV} is a straightforward adaptation of Theorem 2.2 and Proposition 3.4 of \cite{CHMT}, and we refer to this paper for the proof. Notice that \eqref{hyp:moment} could be weakened if one can ensure that this implies
\[		\sup_{K\in \N^*} \E\big( \sup_{t\in [0,T]}\langle Z^K_t,|h|^2\rangle^{1+\epsilon/2}\big)<+\infty.\]

The more general Assumption \eqref{hyp:moment} is made because we stick here with a general growth rate $h$ satisfying ($H.c$). Particular cases where $h$ is explicit and differentiable can be treated directly.

	\subsection{Duality properties for the infinitesimal generator and the transition semigroup of the underlying path process}

\noindent Before considering the stationary solutions of PDE \eqref{limit-moving} and the ancestral path of an individual chosen at random in the population at a given time $T>0$, we study the stochastic process $(X_t)_{t \ge 0}$ describing the change in time of the trait of a given individual with initial value $x\in \R$. The process $(X_t)_{t\geq 0}$ follows the stochastic differential equation
\begin{align}\label{def:X}
    X_t= & x + \rho t + \int_0^t \int_{\R_+}\int_{\R} (y-X_{s_-}) \ind_{\theta\leq \gamma m(X_{s_-},y)} Q(ds,d\theta,dy),
\end{align}
where $Q$ is a Poisson point measure on $\R\times \R_+\times \R$ with intensity the Lebesgue measure. We define on the same probability space the stochastic process $(X^*_t)_{t\geq 0}$ as solution of
\begin{align}\label{def:Xstar}
    X^*_t= & x - \rho t + \int_0^t \int_{\R_+}\int_{\R} (y-X^*_{s_-}) \ind_{\theta\leq \gamma m(y,X^*_{s_-})} Q(ds,d\theta,dy).
\end{align}
Note that the transport terms are opposite and that the jump kernels are dual in some $\mathbb{L}^1$-setting.
These processes are both Markov processes, with transition semigroups respectively $(P_t, t\ge 0)$ and $(P^*_t, t\ge 0)$. For bounded and measurable functions $f$, they are given by
\begin{equation}P_tf(x)=\E_x\big(f(X_t)\big),\qquad \mbox{ and }\qquad P^*_tf(x)=\E_x\big(f(X^*_t)\big).\label{def:Pt-PTstar}\end{equation}
The number of jumps of the process $X$ between $0$ and $t$ follows a Poisson distribution of parameter $\gamma$. Conditionally on this number, the jump times are distributed as the order statistic of a vector of independent uniform random variables on $[0,t]$. Summing over these jumps, we can write
\begin{align}
  P_tf(x) &= \E_x\big(f(X_t)\big)  =  e^{-\gamma t} \, \sum_{k\ge 0} \frac {(\gamma t)^k}{k!}\E\big(f(x+\rho t + U_1+\ldots + U_k)\big),\label{ecriture:Ptf}
\end{align}
where $U_1,\ldots , U_k$ are the jump steps of the $k$ jumps (whose laws depend on $x$). A similar expression with a drift $-\rho$ is available for the process $(X^*_t)_{t\geq 0}$.\\

The aim of this part is to prove the following theorem.

\begin{thm}
\label{duality-sg}
We can extend $P$ and $P^*$ respectively to $\mathbb{L}^\infty$ and $\mathbb{L}^1$. They satisfy the duality relation
\begin{align}\label{duality:PPstar}
    \langle P_tf, g\rangle =  \langle f, P^*_t g\rangle, \qquad \forall f\in \mathbb{L}^\infty, g\in \mathbb{L}^1.
\end{align}
\end{thm}

\medskip The proof will be deduced from a succession of lemmas and remarks.

\bigskip Using It\^o's formula for jump processes with drift \cite[Th. 5.1, page 66]{ikedawatanabe}, it is easy to prove that the domain of the  infinitesimal generators of $X$ and $X^*$ contains  at least the functions of $C^1_b$. Further, we have that for $f\in C^1_b$,  $g\in C^1_b$,
	\begin{equation}\label{def:L-2}
	Lf(x)=\drift\ f'(x)+\gamma \int_{\mathbb{R}}(f(y)-f(x))\ m(x,y)dy,
	\end{equation}
	and
		\begin{equation}\label{def:Lstar}
	L^*g(x)=-\drift\ g'(x)+\gamma \int_{\mathbb{R}}(g(y)-g(x))\ m(y,x)dy.
	\end{equation}

Using the integration by parts formula allows to prove easily  that these generators are in duality for functions of $C^1_b$. We easily extend $L^*$ to the functions of $\mathbb{L}^1$.

\begin{lem}
The generator $L^*$ can been extended in such a way that $L$ and $L^*$ are in duality as follows: for $f\in C^1_b$ and $g\in  \mathbb{L}^1$,
\begin{equation}\label{duality-gen}
    \langle Lf,g\rangle = \langle f,L^* g\rangle.
\end{equation}
\end{lem}

\begin{proof}
 By integration by part, the generator $L$ is associated to the adjoint $L'$ by the following relation: for any $f\in  C^1_b$ and $g\in  \mathbb{L}^1$, $  \langle Lf,g\rangle = \langle f,L'g\rangle$, with
   $$ L'g(x)=-\drift\ {\partial\over \partial x}g(x)+\gamma \int_{\mathbb{R}}(g(y)-g(x))\ m(y,x)dy$$
   where ${\partial\over \partial x}g$ is understood in the distribution sense. Indeed since $g\in  \mathbb{L}^1$, it converges to $0$ at infinity.
   In particular $L'$ and $L^*$ coincide on $C^1_b$ and we will keep the notation $L^*$ for the operator defined on $\mathbb{L}^1$.
\end{proof}

Let us now prove Theorem \ref{duality-sg}. \\
\begin{proof}[Proof of Theorem \ref{duality-sg}]We proceed in two steps. First, we define the adjoint $P'_t$ of $P_t$, and then we prove that it is $P^*_t$. \\

\noindent \textbf{Step 1:} Let $t>0$. From \eqref{ecriture:Ptf}, we can check that the semigroup $P_t$ defines an operator from $\mathbb{L}^\infty$ into $\mathbb{L}^\infty$. It is known (e.g. \cite[IV.3.C page 65]{brezis}) that the dual of $\mathbb{L}^\infty$ is strictly larger (for the inclusion) than $\mathbb{L}^1$. Let $P'_t$ be the adjoint of $P_t$ on the dual space $(\mathbb{L}^\infty)'$ of $\mathbb{L}^\infty$. The domain of $P'_t$ is
\[D(P'_t):=\{\mu \in (\mathbb{L}^\infty)',\ \exists c\geq 0,\ \forall f\in \mathbb{L}^\infty,\ |\langle \mu, P_t f\rangle \leq c \|f\|_\infty\}.\]
Since for any $g\in \mathbb{L}^1$, $|\langle g, P_tf\rangle| \leq \|g\|_1 \|f\|_\infty  $, we see that $\mathbb{L}^1\subset D(P'_t)$ and $P'_t g$ is well defined. For any $f\in \mathbb{L}^\infty$,
\[\langle f, P'_t g\rangle=\langle P_t f,g\rangle.\]
Choosing $f=\mbox{sign}(P'_t g)$, we obtain that
\begin{equation}
    \int_\R |P'_t g(x)|\ dx = \langle P_tf,g\rangle \leq \|g\|_1,\label{etape:P'tg_in_L1}
\end{equation}so that $P'_t g\in \mathbb{L}^1$ and $\||P'_t\|| \leq 1$.\\

\noindent \textbf{Step 2:} Now, our purpose is to prove that $P'_t=P^*_t$ where $P^*_t$ has been defined in \eqref{def:Pt-PTstar}. Let us first prove that $P^*_t$ sends $\mathbb{L}^1$ to $\mathbb{L}^1$. Let $g\in \mathbb{L}^1$. Summing on the number of jumps for the process $X^*$ between $0$ and $t$, we can write  that
\begin{align*}
    \int_\R |P_t^*g(x)|dx &= \int_\R |\E_x\big(g(X^*_t)\big)|dx\\
   & \le  e^{-\gamma t} \, \sum_{k\ge 0} \frac {(\gamma t)^k}{k!} \int_\R |\E_x\big(g(x-\rho t + U_1+\ldots + U_k)\big)|\,dx,
\end{align*}
where $U_1,\ldots , U_k$ are the jump steps of the $k$ jumps.
To simplify notation, let us consider the case of one jump. Using the fact that, conditionally on the number of jumps in the time interval $[0,t]$, the jump times are uniformly distributed on $[0,t]$,  we obtain
\begin{align*}
  \int_\R \Big|\E_x\big(g(x-\rho t + U_1\big)\Big|dx &= \int_\R \Big|\frac{1}{t}\int_0^t \int_\R g(y -\drift(t-t_1)) m(y, x-\drift t_1)\,dy\, dt_1\Big| \,dx   \\
  &\leq  \frac{1}{t} \int_0^t \int_\R \Big|g(y -\drift(t-t_1))\Big| \bigg(\int_\R m(y, x-\drift t_1)\,dx\bigg)\, dy\, dt_1 \\
  &\leq  \frac{1}{t} \int_0^t \int_\R\Big|g(y -\drift(t-t_1))\Big|\,dy\, dt_1 \\
  &\leq   \,\|g\|_{1},
\end{align*}
where we have used Assumption (H.d), i.e. that $\int_\R m(y,x)dx =1$. The same estimate can be obtained for the other terms $ \int_\R |\E_x\big(g(x-\rho t + U_1+\ldots + U_k)\big)|\,dx$, which implies that
$$\|P_t^*g\|_1\leq  \,\|g\|_{1}$$
so that $\||P^*_t\|| \leq 1$.

\noindent \textbf{Step 3:} Let us now consider $f\in C^1_b$, $g\in \mathbb{L}^1\cap C^1_b$ and $t>0$.
We define the function $\varphi(x,s)=g(x-\rho(t-s))$. To ease the following computation, let us give a name to the jump operator in $L^*$ \eqref{def:Lstar}:
\[Jg(x)=\gamma \int_\R \big(g(y)-g(x)\big) m(y,x)dy.\]

It is easy to prove using ($H.d$) that if $g\in \mathbb{L}^1$, then $Jg\in \mathbb{L}^1$ and
$$\|Jg\|_{1} \leq 2 \gamma \|g\|_{1}.
$$

On the one side, $P_t^*g(x)= \E_x\big(g(X^*_t)\big)=\E_x\big(\varphi(X^*_t,t)\big)$, and using Itô's formula:
\begin{align}
    \E_x\big(\varphi(X^*_t,t)\big) 
    = & g(x-\rho t)+\int_0^t \E_x\Big(L^*\varphi(.,s)(X^*_s)+\rho g'(X^*_s-\rho(t-s))\Big) ds\nonumber\\
    = & g(x-\rho t)+\int_0^t P^*_s \big(J \varphi(.,s)\big)(x) ds.\label{etape7}
\end{align}
On the other side, using the definition of $P'_t$, $\langle f, P'_t g\rangle= \langle P_t f,g\rangle$. Our purpose is to develop the right hand side to have an expression similar to \eqref{etape7}. Let us introduce $\psi(s)=P_sf(x)g(x-\rho(t-s))$. Notice that $P_tf(x)g(x)=\psi(t)$. By the Kolmogorov formula:
\begin{equation}
    P_t f(x)=  f(x)+\int_0^t L P_s  f(x)\ ds,\label{Kolmogorov:Pt}
\end{equation}and the fact that $g(x)=g(x-\rho t)+\int_0^t \rho g'(x-\rho(t-s))\ ds$, we obtain:
\begin{align}
    \langle P_tf,g\rangle=  & \int_\R f(x)g(x-\rho t) dx\nonumber\\
     & \hspace{1cm}+\int_0^t \int_\R \Big( LP_sf(x)g(x-\rho(t-s)) +\rho P_sf(x) g'(x-\rho(t-s)) \Big) dx\ ds\nonumber\\
    = & \int_\R f(x)g(x-\rho t) dx+\int_0^t \int_\R P_sf(x) \big( L^*\varphi(.,s)(x)+\rho  g'(x-\rho(t-s)) \big) dx\ ds\nonumber\\
    = & \int_\R f(x)g(x-\rho t) dx+\int_0^t \int_\R f(x) P'_s \big( J\varphi(.,s)\big)(x) dx\ ds.\label{etape8}
\end{align}Using \eqref{etape7}, and since \eqref{etape8} holds for every $f\in C^1_b$, we obtain for almost every $x$,
\begin{align*}
    P^*_t g(x)= & g(x-\rho t) + \int_0^t P^*_s \big( J\varphi(.,s)\big)(x)\ ds\\
    P'_t g(x)= & g(x-\rho t)+ \int_0^t  P'_s \big( J\varphi(.,s)\big)(x) \ ds.
\end{align*}
From this, we deduce that:
\begin{align}
    \|P'_t g-P^*_t g\|_1= & \int_\R \big|P'_t g(x)-P^*_t g(x)\big|\ dx\nonumber\\
    \leq  &
    \int_0^t \int_\R \big|P'_s \big( J\varphi(.,s)\big)(x) -P^*_s \big( J\varphi(.,s)\big)(x)\big|\ dx\ ds\nonumber\\
    \leq  &
   \int_0^t \||P'_s -P^*_s\|| \ \|J\varphi(.,s)\|_1\  ds\nonumber\\
    \leq & \int_0^t 2 \gamma\ \||P'_s -P^*_s\|| \ \|g\|_1 \ ds.
\end{align}
 {This implies that
$$\||P'_t -P^*_s\||   \leq  2\gamma
   \int_0^t  \||P'_s -P^*_s\||\ ds.$$}
 {Now applying Gronwall's inequality, we have that  $\||P'_t - P^*_t\||=0$.} As a consequence, we have for all $f\in C^1_b$ and all $g\in \mathbb{L}^1\cap C^1_b$ that:
\begin{equation}\label{eq:dualitePPstar}
    \langle P_t f, g\rangle= \langle f, P_t^* g\rangle.
\end{equation}
Is is now standard to extend this identity to $f\in \mathbb{L}^\infty$ and $g\in \mathbb{L}^1$. The theorem is proved.
\end{proof}

\subsection{Duality and stationary solution of the limiting PDE}\label{ssec:stationary}


	

{First, we start with a sufficient condition ensuring that the solution of the limiting PDE \eqref{limit-moving} is a function $\xi$ whose values $\xi_t$ are absolutely continuous measures with respect to the Lebesgue measure on $\R$. Exploiting the duality relations, we show that the densities $f_t$ of the measures $\xi_t$ solve a PDE with $L^*$.}

\begin{prop}\label{prop:density}
If the measure $\xi_0$ admits a {non negative} density $f_0$ with respect to the Lebesgue measure on $\R$, then so does $\xi_t$ for all $t>0$. The densities $f_t$ for $t>0$ define a solution in $\Co([0,T],L^1(\mathbb{R}))$ of:
	\begin{equation}\label{eq:pde}
		\partial_{t}f_{t}(x)=L^{\ast}f_{t}(x) +  \left(h(x)-\int_{\mathbb{R}}f_{t}(y)dy  \right)f_{t}(x).
	\end{equation}
\end{prop}
	
\begin{proof}
The idea of the proof is the following: if \eqref{eq:pde} possesses a solution $f_{t}$ in $C([0,T],\mathbb{L}^{1})$, then $f_{t}(x)dx$ is solution of \eqref{limit-moving}, and the identification $\xi_t=f_{t}(x)dx$ follows from the uniqueness of the solution of \eqref{limit-moving}. Thus, we only have to prove that \eqref{eq:pde} possesses a solution with initial condition $f_{0}$. To prove this, we follow closely the computation in \cite{dyson2000nonlinear}. The proof is detailed in Appendix \ref{app:density}. \end{proof}

	Now, let us discuss on the stationary solutions of \eqref{limit-moving}. If we assume to have in hand a non-negative stationary solution $F\in \mathbb{L}^1$ of \eqref{limit-moving}, it necessarily satisfies
	\[
	L^{\ast}F +  hF=\lambda F,
	\]with $\lambda=\int_{\mathbb{R}}F(x)\ dx$. Thus, $F$ is an eigenvector of the operator $\mathcal{A}$ defined by
	\begin{equation}\label{def:Acal}
	\mathcal{A}f=L^{\ast}f+hf.
	\end{equation}
Reciprocally, we can state the following result.

		\begin{prop}
		\label{prop:stationary}Under the assumptions $(H.a)$ and $(H.b)$,
		there exists a unique (up to a multiplicative constant) positive eigenvector $G\in L^{1}(\mathbb{R})$ for $\mathcal{A}$ associated to an eigenvalue $\lambda$. If $\lambda$ is positive, then, $F=\lambda G/\|G\|_1$ is the unique non-trivial {positive}  stationary solution to \eqref{limit-moving}.
	\end{prop}
\begin{proof}
Existence and uniqueness of a positive eigenvector $G$ have been proven by Cloez and Gabriel \cite{CG}. The positiveness of the eigenvalue $\lambda$ associated with the non-negative eigenvector is not stated in the result of \cite{CG}.
If $\lambda>0$, the stationary solution of \eqref{limit-moving} can be obtained by defining $F=\lambda G/\|G\|_1$: we have that $\|F\|_1=\lambda$ and by linearity $L^*F+hF=\lambda F$.
\end{proof}

{In the Gaussian case \cite{CHMT}, the condition for the positiveness of $\lambda$ could be expressed in terms of the model parameters. This was feasible as the stationary distribution $F$ was an explicit Gaussian distribution. In the present situation, only sufficient conditions can be derived such as the one stated in the following lemma.
\begin{lem} Assume that in addition to (H.a) and (H.b),  there exists $x_1>0$ such that $\inf_{x\in(-x_1,x_1)}h(x)  >\gamma$ and that we can choose the constants $x_1$, $\varepsilon$ and $\kappa_0$ such that $\gamma \kappa_0\varepsilon^3 \geq 12 \rho x_1$. Then, $\lambda>0$.
\end{lem}
}

\begin{proof}
Following \cite{CG}, let us consider the function
$
\psi_0(x)=\left(1-\frac{x^2}{x_1^2}\right)_{+}^{2}.
$
We have:
\begin{multline}
    (L+h)\psi_0(x) \label{etape3}\\
    \begin{aligned}= & - 4\rho \frac{x}{x_1^2} \left(1-\frac{x^2}{x_1^2}\right)_{+}+\gamma \int_R \big(\psi_0(y)-\psi_0(x) \big)m(x,y)dy + h(x)\psi_0(x)\\
    = & - 4\rho \frac{x}{x_1^2} \left(1-\frac{x^2}{x_1^2}\right)_{+}+\gamma \int_{-x_1}^{x_1} \left(1-\frac{y^2}{x_1^2}\right)^{2}m(x,y)dy + \big(h(x)-\gamma\big)\psi_0(x)\\
    \geq & \big(h(x)-\gamma)\psi_0(x) -\frac{4\rho}{x_1} \left(1-\frac{x^2}{x_1^2}\right)_{+}+\gamma \kappa_0 x_1 \int_{-1}^1 \big(1-z^2\big)^{2} \ind_{\big(\frac{x-\varepsilon}{x_1},\frac{x+\varepsilon}{x_1}\big)}(z)dz
\end{aligned}
\end{multline}by using Assumption $(H.a)$. Without loss of generality, it can be assumed that $0<\varepsilon<x_1$. When $x\in [-x_1,x_1]$, the integral in the third term is lower bounded by
\begin{align}\int_{1-\varepsilon/x_1}^1 (1-z^2)^2 dz= & \frac{\varepsilon}{x_1}\Big[\frac{8}{15}-\frac{7}{15}\Big(\big(1-\frac{\varepsilon}{x_1}\big)+\big(1-\frac{\varepsilon}{x_1}\big)^2\Big)+\frac{1}{5}\Big(\big(1-\frac{\varepsilon}{x_1}\big)^3+\big(1-\frac{\varepsilon}{x_1}\big)^4\Big)\Big]\nonumber\\
= & \frac{20}{15}\frac{\varepsilon^3}{x_1^3}-\frac{\varepsilon^4}{x_1^4}+\frac{1}{5}\frac{\varepsilon^5}{x_1^5} \geq \frac{1}{3}\frac{\varepsilon^3}{x_1^3}.\label{etape4}
\end{align}
Gathering \eqref{etape3} and \eqref{etape4}, we obtain that
\begin{equation*}
    (L+h)\psi_0 \geq \inf_{x\in(-x_1,x_1)}h(x) \psi_0 -\gamma \psi_0 -\frac{4\rho}{x_1} \left(1-\frac{x^2}{x_1^2}\right)_{+}+ \frac{\gamma \kappa_0\varepsilon^3}{3x_1^2} .
\end{equation*}
Under our hypotheses,
\begin{equation}
    -\frac{4\rho}{x_1} \left(1-\frac{x^2}{x_1^2}\right)_{+}+ \frac{4\rho}{x_1} \geq 0,
\end{equation}so that $(L+h)\psi_0\geq \beta_0 \psi_0$ with
\begin{equation}\beta_0=\inf_{x\in(-x_1,x_1)}h(x)  -\gamma>0.
\end{equation}


It follows from differentiating the semigroup $S_{t}$ associated with $L+h$ acting on $C_{0}(\mathbb{R})$ that
\[
S_{t}\psi_{0}\geq \psi_{0}e^{\beta_{0}t}.
\]
Now, according to our application of Theorem 2.1 of \cite{CG}, we have that for the semigroup $S^{\ast}$ associated to $\mathcal{A}$ \eqref{def:Acal},
\[
e^{-\lambda t}S^{\ast}_{t}f\xrightarrow[t\to\infty]{}F\text{ in }\mathbb{L}^1.
\]
Hence, for any positive $f\in \mathbb{L}^{1}$,
\[
\langle \psi_{0}, F\rangle=\limsup_{t\to\infty}\ \langle \psi_{0}, e^{-\lambda t}S^{\ast}_{t}f\rangle=\limsup_{t\to\infty}e^{-\lambda t}\langle S_{t}\psi_{0}, f\rangle\geq \limsup_{t\to\infty}e^{(\beta_{0}-\lambda) t}\langle \psi_{0}, f\rangle.
\]
Thus, $\beta_{0}-\lambda\leq 0$, which gives the result since $\beta_{0}>0$.
\end{proof}


It is important to note that $F$ is also a solution of a linearized version of \eqref{eq:pde}
\begin{equation}\label{eq:linpde}
		\partial_{t}f_{t}(x)=L^{\ast}f_{t}(x) + h^{\lambda}(x)f_{t}(x),
	\end{equation}
	where
	\begin{equation}\label{hlambda}	    h^{\lambda}(x)=h(x)-\lambda.
	\end{equation}
{Note that this notation $h^\lambda$ will be extensively used in  what follows. }\\

Let us prove in the next lemma that Assumptions ($H.a$) to ($H.c$)  ensure that $F$ has a finite $2q$-moment, compatible with the assumptions \eqref{hyp:moment} for the initial conditions of the population process.

	
	\begin{lem}Assume ($H.a$) to ($H.c$) and assume that $\lambda>0$. Then, $F$ has a finite moment of order $2q$, i.e.
		\[
		\int_{\mathbb{R}}x^{2q}F(x)\ dx<+\infty .
		\]
	\end{lem}
	\begin{proof}
		Set, for any $x\in\mathbb{R}$ and any $n\in \mathbb{N}$,
		\[
		g_{n}(x)=\frac{1}{1+\frac{x^{2q+2}}{n}}.
		\]
		Note that $g_n$ is non-negative, and that there exists a constant $C\in\mathbb{R}_{+}$ (that can be chosen independent of $n$), such that
		$|g'_{n}|\leq C g_{n}$.
		Obviously, $x^{2q} g_{n}\in D(L)$, which implies that
		\[
		\lambda\langle x^{2q} g_{n}, F\rangle=\langle x^{2q} g_{n}, L^{\ast}F+hF\rangle=\langle L(x^{2q}g_{n})+h x^{2q} g_{n}, F\rangle.
		\]
		Hence,
		\begin{align}
			\lambda\langle x^{2q}g_{n}, F\rangle\leq &  \rho\langle 2q |x|^{2q-1} g_{n}(x), F\rangle+\rho\left|\langle x^{2q}g'_{n}(x), F\rangle\right|+\gamma\,m_{2q}-\gamma\langle x^{2q}g_{n}(x), F\rangle+ \langle hx^{2q}g_{n}, F\rangle\nonumber\\
			\leq  & \int_{\mathbb{R}}\left(|x|^{2q-1} g_{n}(x)\left(2q \rho+|x|\left(C\rho +\gamma+h(x)\right)\right)F(x)\right)\ dx+\gamma m_{2q},
		\end{align}
		where
		\[
		{m_{2q}=\sup_{x\in\mathbb{R}}\int_{\mathbb{R}}y^{2q}\ m(x,y)\ dy<+\infty,}
		\]by Assumption \eqref{hyp:moment-q}. Now, using Assumption ($H.c$), there exists a compact set $K$ of $\mathbb{R}$ such that
		\[
		\int_{\mathbb{R}\setminus K}\left(|x|^{2q-1} g_{n}(x)\left(2q\rho +|x|\left(C\rho +\gamma+h(x)\right)\right)F(x)\right)\ dx \leq 0,
		\]
		which implies, with the constant $c$ appearing in ($H.b$):
		\begin{equation}
			\lambda\langle x^{2q} g_{n}, F\rangle\leq\int_{K}\left(|x|^{2q-1}g_{n}(x)\left(2q\rho +|x|\left(C\rho +\gamma+c\right)\right)F(x)\right)\ dx+\gamma m_{2q}.
		\end{equation}
		Now, Fatou's lemma and Lebesgue convergence theorem imply that
		\begin{align*}
			\lambda \int_{\mathbb{R}}x^{2q}F(x)\ dx \leq & \liminf_{n\to + \infty} \lambda \int_{\mathbb{R}}x^{2q} g_n(x) F(x)\ dx \\
			\leq & \liminf_{n\to + \infty}\int_{K}\left(|x|^{2q-1}g_{n}(x)\left(2q\rho+|x|\left(C\rho +\gamma+c\right)\right)F(x)\right)\ dx+\gamma m_{2q}\\
			= & \int_{K}\left(|x|^{2q-1}\left(2q\rho  +|x|\left(C\rho+\gamma+c\right)\right)F(x)\right)\ dx+ \gamma m_{2q}<+\infty,
		\end{align*}by \eqref{hyp:moment-q}.
	\end{proof}

	\subsection{Coupled population processes}\label{section:coupling}

When we start from the initial condition $F$, the non-linear competition term $\langle \xi_t,1\rangle=\| F\|_1=\lambda$ remains constant in time. Considering initial conditions close to $F$ for the process will lead to non-linear terms close to $\lambda$. This remark is the basis of the coupling with two other individual-based random processes $(\widetilde{Z}^K_t)_{t\in \R_+}$ and $(\widetilde{H}^{K}_{t})_{t\in \mathbb{R}_{+}}$.
These processes are pathwisely  defined by similar equations as \eqref{eq:eds_Poisson} and \eqref{eq:eds_Poisson-H} (Appendix A) for $Z^K$ and $H^K$ but where the non-linear competition term {$N^{K}_{t}/K$} has been replaced by the expected limiting competition rate $\lambda$.
When $K$ tends to infinity and  if the initial conditions of both processes converge to $F$, then $Z^K$ and $\widetilde{Z}^K$, and $H^K$ and $\widetilde{H}^K$ will be close.
The next proposition states  a precise approximation result linking $\widetilde{H}^{K}$ and  $H^{K}$ when $K$ goes to infinity and  $Z^K_0$ converges to $F$.
	
	\begin{prop}\label{prop:coupligHHtilde}
		Assume that \eqref{hyp:moment} holds and that $\,{Z}^{K}_{0}\xrightarrow[K\to\infty]{w} F\,$. Then for any  continuous and bounded function $\Phi$ on $\D$,
		$$ \lim_{K\to+\infty}\E(\sup_{t\le T} |\langle H^{K}_{t},\Phi\rangle-\langle\widetilde{H}^{K}_{t},\Phi\rangle|^2) = 0.$$
	\end{prop}
The proof is similar to the proof of Proposition 3.4 in \cite{CHMT}, to which we refer.\\

The  the processes $\widetilde{Z}^K$ and $\widetilde{H}^K$ are linear birth and death processes and satisfy the branching property. Therefore, the spinal techniques as developed by \cite{bansayedelmasmarsalletran,hardy2006new,harris1996large,Marguet1,Marguet2} can be used. Note also that  it is sufficient to consider the processes started from a unique individual. In the sequel, we will denote by $\widetilde{Z}$ the branching process $K\widetilde{Z}^K$ started from a single individual of trait $x$: $\widetilde{Z}_0=\delta_{x}$.

		\section{Linear case : Feynman-Kac formula  and spinal process}\label{section:linear}
		
			In this part we only focus on the processes $\widetilde{Z}^K$ and $\widetilde{H}^K$.  It is possible to summarize their intensity measures  with a single process by mean of  many-to-one formulas \cite{Marguet2}. We will next use these results to approximate the distribution of a typical lineage for the original population process.

\subsection{Feynman-Kac formula and  law of the spinal process}

	\begin{lem}
		\label{lem:MTO}
		Let $\varphi$ in $C_{b}(\mathbb{R})$. Then, for any positive time $t$, for any $x\in \mathbb{R}$, we have
		\begin{equation}
			\label{eq:MTO}
			\mathbb{E}_{\delta_{x}}\left[\langle \widetilde{Z}_{t}, \varphi\rangle \right]=\mathbb{E}_{x}\left[\exp\left(\int_{0}^{t} h^{\lambda}(X_s)\ ds \right)\varphi(X_{t}) \right]=:\widehat{P}_t\varphi (x),
		\end{equation}
		where $X$ is the  process defined in  \eqref{def:X}.
	\end{lem}
	
	\begin{proof}
		Let us give a simple  proof based on It\^o's formula. Let us first note that the intensity measure of $\widetilde{Z}_t$, $\,
		\nu_{t}(dy)=\mathbb{E}_{\delta_{x}}\left[\widetilde{Z}_{t}(dy)\right]\,$ defined for any $\varphi$ in $C_{b}(\mathbb{R})$ by
		\[
		\langle\nu_{t}, \varphi \rangle =\mathbb{E}_{\delta_{x}}\left[\langle \tilde{Z}_{t}, \varphi\rangle \right]
		\]
		is  the unique weak solution of
		\begin{equation}
			\label{eq:prob2}
			\begin{cases}
				\partial_{t}\nu_{t}=L'\nu_{t}+ h^{\lambda}(x)\nu_{t}(dx),\\
				\nu_{0}=\delta_{x},
			\end{cases}
		\end{equation}
		where $L'$ is the adjoint of $L$.
		Indeed, taking expectation in \eqref{ref:partSys} immediately shows that $\nu$ is weak solution of \eqref{eq:prob2}. Uniqueness of such a solution is proven as in Theorem \ref{thm-CV} (see Th.2.2 in \cite{CHMT}).

	 Let us now show that the r.h.s.\ term of \eqref{eq:MTO} also  satisfies \eqref{eq:prob2}. Uniqueness will yield the result.
		Let $\varphi$ in $C^1_{b}(\mathbb{R})$. It is known since $\varphi$ is in the extended domain of $L$ that
		\[
		M_{t}=\varphi(X_{t})-\varphi(X_{0})-\int_{0}^{t}L\varphi(X_{s})\ ds
		\]
		is a martingale.
		Thus applying It\^o's formula with jumps (e.g. \cite[Th.5.1]{ikedawatanabe}) to the semimartingale
		$\ \exp\left(\int_{0}^{t} h^{\lambda}(X_s)ds \right)\varphi(X_{t})$,
		we have
		\begin{multline}
		\exp\left(\int_{0}^{t} h^{\lambda}(X_{s})ds \right)\varphi(X_{t})
		=\varphi(X_{0})+\int_{0}^{t}\exp\left(\int_{0}^{s} h^{\lambda}(X_u)du\ \right) dM_{s}\\
		+\int_{0}^{t}\exp\left(\int_{0}^{s} h^{\lambda}(X_u)du\right) L\varphi(X_{s})\ ds
		+\int_{0}^{t}\varphi(X_{s})h^{\lambda}(X_s)\exp\left(\int_{0}^{s} h^{\lambda}(X_u)du \right)\ ds\end{multline}
		By Assumption $(H.b)$, the stochastic integral with respect to $dM_s$ defines a square integrable martingale and by taking the expectation, we obtain that
		\begin{multline} \label{ito}
			\mathbb{E}_{x}\left[\exp\left(\int_{0}^{t} h^{\lambda}(X_s)ds \right)\varphi(X_{t}) \right] =\varphi(x)\\
			+  \mathbb{E}_{x}\bigg[\int_{0}^t \exp\left(\int_{0}^{s} h^{\lambda}(X_u)du \right)  \bigg\{h^{\lambda}(X_s)\,\varphi(X_{s}) + L\varphi(X_{s})\bigg\}ds\bigg].
			\end{multline}
		
		If we define the measure $\mu_{t}$ for any test function $\varphi\in \Co^1_b(\R)$ by
		$$\langle \mu_{t}, \varphi\rangle = \mathbb{E}_{x}\left[\exp\left(\int_{0}^{t} h^{\lambda}(X_s)ds \right)\varphi(X_{t}) \right] ,$$
		we obtain from \eqref{ito}  that
		$$\langle \mu_{t}, \varphi\rangle = \langle \delta_{x}, \varphi\rangle +\int_{0}^t \langle \mu_{s}, h^{\lambda}\varphi+L\varphi \rangle ds.$$
		That proves that the flow $(\mu_{t}, t\ge 0)$ is a weak solution of \eqref{eq:prob2} and the conclusion follows by uniqueness of Theorem \ref{thm-CV}.
	\end{proof}

	The previous many-to-one formula  characterizes the law  of $\tilde Z_{t}$. It can be extended to the whole trajectory (see \cite{CHMT,Marguet1}).
	
	\begin{lem}\label{lem:cvHtilde}
		We have that for $T>0$, $\Phi:\mathbb{D}([0,T],\mathbb{R})\to\mathbb{R}$ a continuous and bounded function and $x\in \mathbb{R}$:
		\begin{align}
			\label{eq:MTO099}
			\mathbb{E}_{\delta_x}\left[\langle \widetilde{H}_T,\Phi\rangle\right]&=\mathbb{E}_{\delta_{x}}\left[\sum_{i\in \widetilde{V}_{T}}\Phi({X}^i_{s},\ s\leq T) \right]=\mathbb{E}_{x}\left[\exp\left(\int_{0}^{T} h^{\lambda}(X_{s})\ ds \right)\Phi(X_{s},\ s\leq T) \right].
		\end{align}

	\end{lem}

	This trajectorial Feynman-Kac formula can be used to characterize the law of an auxiliary process that will help us to understand the typical lineage later.
	Let us introduce, for all $x\in\mathbb{R}$ and $t\geq 0$, the expected population mass, defined  by
	\begin{equation}\label{eq:mt}
		m_t(x)=\E_{\delta_x}\big[\langle \widetilde Z_t,1\rangle\big]=\E_{\delta_x}\left[\langle \widetilde{H}_t,1\rangle \right]
		= \mathbb{E}_{x}\left[\exp\left(\int_{0}^{t} h^{\lambda}(X_{s})\ ds \right) \right].
	\end{equation}
	Thus, we use the r.h.s.\ of \eqref{eq:MTO099} to define a family of probability measures  $\mu_{x}^{T}$ on $\mathbb{D}([0,T],\mathbb{R})$ by
	\begin{equation}\label{def:muTx}
	\mu_{x}^{T}(A)=\frac{\E_{\delta_x}\left[\langle \widetilde{H}_T,\ind_A\rangle \right]}{\E_{\delta_x}\left[\langle \widetilde{H}_T,1\rangle \right]}
	= \frac{1}{m_{T}(x)}\mathbb{E}_{x}\left[\exp\left(\int_{0}^{T} h^{\lambda}(X_{s})\ ds \right)\mathds{1}_{X_.\in A} \right]
	\end{equation}
	for any measurable subset $A$ of $\mathbb{D}([0,T],\mathbb{R})$. For a probability measure $\nu$, let us also define
	\begin{equation}\label{def:muTnu}
		\mu^{T}_{\nu}(A)=\int_\R \mu^T_x(A) \ \nu(dx).
		\end{equation}
		
	Our next proposition characterizes the law of the underlying process (forward in time).
	\begin{prop}
		\label{prop:SGmu}
		The distribution $\mu^T_x$  is the one of a time inhomogeneous Markov process $Y$ issued from $x$ and  with semigroup $(\widetilde{P}_{s,t})_{t\geq s\geq 0}$ given for a bounded continuous function $\varphi$  by
		\begin{equation}
			\label{eq:spineSG}
			\widetilde{P}_{s,t+s}\varphi(x)=\cfrac{\widehat{P}_{t}(\varphi m_{T-t-s})(x)}{m_{T-s}(x)}.
		\end{equation}
\end{prop}
	\begin{proof}
Let $\varphi$ be a some test function, and assume that $s\leq t$ are such that $s+t\le T$. Denoting by $\mathcal{F}=(\mathcal{F}_{s})_{s\geq 0}$ and $\mathcal{F'}=(\mathcal{F'}_{s})_{s\geq 0}$ the natural
filtrations associated respectively to $Y$ and $X$, our aim is to prove that: 	
		\begin{equation}
			\label{eq:lawY-2}
			\mathbb{E}_x\left[\varphi(Y_{t+s})\mid \mathcal{F}_{s}\right]=\widetilde{P}_{s,t+s}\varphi(Y_{s}).
		\end{equation}

		Now, for a $\mathcal{F}_{s}$-measurable random variable $\Psi(Y_{u},\ u\leq s)$, we have
		\begin{multline}
			\label{eq:condExp}
		\mathbb{E}\left[\Psi(Y_{u},\ u\leq s) \varphi(Y_{t+s})\right]\\
		\begin{aligned}
		= &  \frac{1}{m_T(x)}\mathbb{E}_{x}\left[\Psi(X_{u},\ u\leq s) \varphi(X_{t+s})e^{\int_{0}^{T}h^{\lambda}(X_{u}) du}\right]\\
		= & \frac{1}{m_T(x)}\mathbb{E}_{x}\Big[\Psi(X_{u},\ u\leq s)e^{\int_{0}^{s}h^{\lambda}(X_{u}) du}\\  
& \hspace{2cm} \times \E\left[ \varphi(X_{t+s})e^{\int_{s}^{t+s}h^{\lambda}(X_{u}) du}\
		\E\left(e^{\int_{t+s}^{T}h^{\lambda}(X_{u}) du}\ |\ \mathcal{F}'_{t+s}\right)\ |\ \mathcal{F}'_s\right]\Big]\\
		= & \frac{1}{m_T(x)}\mathbb{E}_{x}\left[\Psi(X_{u},\ u\leq s)e^{\int_{0}^{s}h^{\lambda}(X_{u}) du}\ \E\left[ \varphi(X_{t+s})e^{\int_{s}^{t+s}h^{\lambda}(X_{u}) du} m_{T-(t+s)}(X_{t+s})\ |\ \mathcal{F}'_s\right]\right]\\
		= & \frac{1}{m_T(x)}\mathbb{E}_{x}\left[\Psi(X_{u},\ u\leq s)e^{\int_{0}^{s}h^{\lambda}(X_{u}) du}\widehat{P}_t \big(\varphi \ m_{T-t-s}\big)(X_s)\right].
		\end{aligned}
		\end{multline}
For the first equality of \eqref{eq:condExp}, we have use the definition of the distribution $\mu^T_x$ of $Y$ (see \eqref{def:muTx}). For the third inequality, we have use the strong Markov property
for $X$ at time $t+s$ with the definition \eqref{eq:mt} of $m_{T-(t+s)}(x)$. For the fourth inequality, we have use again the strong Markov property for $X$ at time $s$ and the definition \eqref{FK-Pt} of $\widehat{P}_t$.  		
		
		Using again the strong Markov property, we have that
\[m_{T-s}(X_{s})=\E_{X_s}\Big(e^{\int_0^{T-s}h^\lambda(X_u)du}\Big)=\E_x\Big(e^{\int_s^T h^\lambda(X_u)du}\ |\ \mathcal{F}'_s\Big)\]
		so that the last term of Equation
		\eqref{eq:condExp} gives
		\begin{align*}
			\begin{split}
				\mathbb{E}\left[\Psi(Y_{u},\ u\leq s) \varphi(Y_{t+s})\right]
		&=\frac{1}{m_T(x)}\mathbb{E}_{x}\left[e^{\int_0^T h^\lambda(X_u)du}\ \Psi(X_{u},\ u\leq s) \frac{\widehat{P}_{t}\left(\varphi m_{T-t-s}\right)(X_{s})}{m_{T-s}(X_{s})}\right]\\
				&=\mathbb{E}_{x}\left[\Psi(Y_{u},\ u\leq s)\frac{\widehat{P}_{t}\left(\varphi m_{T-t-s}\right)(Y_{s})}{m_{T-s}(Y_{s})}\right]
			\end{split}
		\end{align*}by using again \eqref{def:muTx} for the last equality.
		
		 \end{proof}

		 \subsection{Duality properties and time reversal of the process $Y$}

In this section, the purpose is to  show that the time-reversal of the process $Y$ is the homogeneous Markov process as announced in Theorem \ref{reversed}.	To do so, we use very general general results for
time reversal of Markov processes (see \cite[Chapter XVIII.46]{DM2}, and reference therein). We need to prove some duality relations.

\subsubsection{Duality between the Feynman-Kac semigroups $\widehat P$ and $\widehat P^*$}
	In Lemma \ref{lem:MTO}, we have proved that the expectation of the branching process is related to a (non Markovian) semigroup based on the multiplicative functional $\exp\big(\int_{0}^{t} h^{\lambda}(X_s)\ ds\big)$ which is bounded by $e^{ct} + 1$ (Assumption $(H.b)$).
	For any function $f\in \mathbb{L}^\infty$, we can define
\begin{align}
\label{FK-Pt}
    \widehat{P}_tf(x)= 	\mathbb{E}_{x}\left(\exp\left(\int_{0}^{t} h^{\lambda}(X_s)\ ds \right)f(X_{t}) \right),
\end{align}
and  $\| \widehat{P}_t f\|_{\infty} \leq (e^{ct} + 1) \|{P}_t f\|_{\infty} \leq (e^{ct} + 1) \|f\|_{\infty}$.
In an analogous way, we may also define the Feynman-Kac semigroup associated with  the process $X^*$:  for any function $g\in \mathbb{L}^1$,
\begin{align}
\label{FK-Pt*}
    \widehat{P}^*_tg(x)= 	\mathbb{E}_{x}\left(\exp\left(\int_{0}^{t} h^{\lambda}(X^*_s)\ ds \right)g(X^*_{t}) \right),
\end{align}
and $\ \| \widehat{P}^*_t g\|_{1} \leq (e^{ct} + 1)\| {P}^*_t g\|_{1} \leq (e^{ct} + 1) \|  g\|_{1}$.

\begin{rem}
Note that the semigroups $\widehat{P}_t$ and $\widehat{P}^*_t$ are not conservative: for example, $\widehat{P}_t \ind (x)=m_t(x)$, see \eqref{eq:mt}. The function $h^\lambda$ does not necessarily have constant sign,
but for the constant $c$ defined in Hypothesis ($H.b$), the rescaled semigroups $e^{-ct}\widehat{P}_{t}$ and $e^{-ct}\widehat{P}^{\ast}_{t}$ are sub-Markovian.
Hence, in the following proof we will work up to this scaling $e^{-ct}$, then assuming that the semigroups $\widehat{P}_t$ and $\widehat{P}^*_t$ are sub-Markovian.
\end{rem}

\medskip
Let us now prove that the duality relation between $P$ and $P^*$ extends to $\widehat{P}$ and $\widehat{P}^*$.

	\begin{lem}
		\label{lem:duality}
		The semigroups $\widehat{P}$ and $\widehat{P}^{\ast}$  satisfy the duality relation:
		\begin{equation}
			\label{eq:duality2}
			\langle \widehat{P}_{t}f,g\rangle=\langle f,\widehat{P}^{\ast}_{t}g\rangle, \quad \forall t\geq0 ,\quad \forall f \in \mathbb{L}^\infty,g\in \mathbb{L}^1.
		\end{equation}
		\end{lem}
		
		\medskip
	\begin{proof}
		For any positive integer $n$ and any $x\in\mathbb{R}$, we define $h_{n}(x)=h^{\lambda}(x)\vee(-n)$. We consider,
the sequences of semigroups $\widehat{P}^{n}$ and $\widehat{P}^{\ast,n}$ defined similarly as \eqref{FK-Pt} and \eqref{FK-Pt*} but using $h_{n}$ instead of $h^{\lambda}$.

		We now show that for any positive integer $n$, the semigroups $\widehat{P}^{n}$ and $\widehat{P}^{\ast,n}$ are in duality (with respect to the Lebesgue measure).
	
		Let $f$ be a measurable positive and bounded function and  $g$ a function in $\mathbb{L}^1$.
		it is straightforward to check that $\widehat{P}^{n}$ and $\widehat{P}^{\ast,n}$ respectively satisfies
	\begin{equation}	\label{eq:mildForm}
		\widehat{P}^{n}_{t}f(x)=\int_{0}^{t}\widehat{P}^{n}_{s}\left(h_{n}P_{t-s}f\right)(x)\ ds+P_{t}f(x)
	\end{equation}
		and
		\begin{equation}
		\label{eq:mildForm*}
		\widehat{P}^{\ast,n}_{t}g(x)=\int_{0}^{t}P^{\ast}_{t-s}\left(h_{n}\widehat{P}^{\ast,n}_{s}g\right)(x)\ ds+P^{\ast}_{t}g(x).
		\end{equation}
	Indeed, for instance, the r.h.s\ of  Equation \eqref{eq:mildForm} rewrites using the Markov property
\begin{multline*}
	\int_{0}^{t}\mathbb{E}_{x}\left[\exp\left(\int_{0}^{s}h_{n}(X_{u})\ du\right)h_{n}(X_{s})\mathbb{E}_{X_{s}}\left[f(X_{t-s})\right]\right]\ ds+\mathbb{E}_{x}\left[f(X_{t})\right]\\
	=\mathbb{E}_{x}\left[\int_{0}^{t}h_{n}(X_{s})\exp\left(\int_{0}^{s}h_{n}(X_{u})\ du\right) ds\ f(X_{t})\right]+\mathbb{E}_{x}\left[f(X_{t})\right]=\widehat{P}^{n}_{t}f(x)
\end{multline*}\\
	Since $\ \widehat{P}^{\ast,n}_{t}g \in \mathbb{L}^1$, we can apply \cite{DMC} Chap IX no 14, to justify the existence of a submarkovian semigroup $Q$ in duality with $\widehat{P}^{\ast,n}$.
		Hence, for $f \in \mathbb{L}^\infty,g\in \mathbb{L}^1$, we have
		$$	\langle Q_{t} f,\ g\rangle=\langle f,\ \widehat{P}^{\ast,n}_{t} g\rangle$$
		from which we deduce, since $h_n$ is bounded, that $\|| Q_t\|| \leq \|h_n\|_{\infty}$. Further, we have
		\begin{align*}
			\langle Q_{t} f,\ g\rangle=\langle f,\ \widehat{P}^{\ast,n}_{t} g\rangle&=\int_{\mathbb{R}}\int_{0}^{t}{P}^{\ast}_{t-s}\left(h_{n}\widehat P^{\ast,n}_{s}g\right)(x)\ ds\ f(x)\ dx+\langle f,\  P^\ast_{t} g\rangle\\
			&=\int_{0}^{t}\int_{\mathbb{R}}{P}^{\ast}_{t-s}\left(h_{n}\widehat P^{\ast,n}_{s}g\right)(x)\ f(x)\ dx\ ds+\langle P_{t}f,\ g\rangle\\
			&=\int_{0}^{t}\langle {P}_{t-s}f, h_{n}\widehat P^{\ast,n}_{s}g\rangle \ ds+\langle P_{t}f,\ g\rangle\\
			&=\int_{0}^{t}\langle Q_{s}(h_{n}P_{t-s}f),\ g\rangle \ ds+\langle P_{t}f,\ g\rangle.
		\end{align*}
		Hence,   for Lebesgue almost every $x\in\mathbb{R}$, $Qf$ satisfies the equation
		\[
		Q_{t}f(x)=\int_{0}^{t}Q_{s}(h_{n}P_{t-s}f)(x)\ ds+P_{t}f(x).
		\]
		Finally, using \eqref{eq:mildForm}, we obtain that
		\begin{multline*}
		\left|\widehat{P}^{n}_{t}f(x)-Q_{t}f(x)\right|\leq \int_{0}^{t}\left|\widehat{P}^{n}_{s}(h_{n}P_{t-s})f(x) - Q_{s}(h_{n}P_{t-s}f)(x)\right|\ ds \\\leq \int_{0}^{t} \|| \widehat{P}^{n}_{s}- Q_s\||\ \|h_{n}P_{t-s}f\|_{\infty} ds\leq \|h_{n}\|_{\infty}\|f\|_{\infty} \int_{0}^{t}\|| \widehat{P}^{n}_{s}- Q_s\||\ ds.
		\end{multline*}
		Thus, for any $t\leq T$,
	$$	\|| \widehat{P}^{n}_{t}- Q_t\|| \leq  \|h_{n}\|_{\infty} \int_{0}^{t}\|| \widehat{P}^{n}_{s}- Q_s\||\ ds,$$
		and Gronwall's lemma allows to conclude that for all $s\leq T$ and all $f\in \mathbb{L}^{\infty}$,
		\[
		Q_{s}f=\widehat{P}^{n}_{s}f.
		\]
	 It follows that the semigroups $\widehat{P}^{n}$ and $\widehat{P}^{\ast,n}$ are in duality. The extension to $\widehat{P}$ and $\widehat{P}^{\ast}$ follows from a monotone convergence argument as $h_{n}$ converges in a monotonous way to $h^{\lambda}$.
	\end{proof}

	\subsubsection{Time-reversal of the  process $Y$}\label{sec:Yreturn}

	The next results are consequences of the duality relationship between $\widehat{P}$ and $\widehat{P}^{\ast}$. The first shows that $\lambda^{-1}m_{t}(x)F(x)$ is a probability density.
	The second result determines the law of  $Y$ at any time when this process is started from $m_{T}(x)F(x)dx$. This initial condition will naturally appear in the next section for the non linear problem.
	
	\begin{lem}Let  $t$ be a fixed positive time. Then, the function $x\in\mathbb{R}\mapsto m_{t}(x)F(x)$ is positive with
		\[
		\int_{\mathbb{R}}m_{t}(x)F(x)\ dx=\lambda.
		\]
	\end{lem}
	\begin{proof}
		Using the duality between $\widehat{P}_t$ and $\widehat{P}^*_t$, we have:
		\begin{align*}
			\langle m_{t},F\rangle=\langle \widehat{P}_{t}1,F\rangle=\langle 1,\widehat{P}^{\ast}_{t}F\rangle=\langle 1,F\rangle=\lambda.
		\end{align*}The third equality comes from the fact that almost everywhere
		\begin{equation}\label{eq:PstarF=F}\widehat{P}^*_t F=F.\end{equation} Indeed, for any function $\varphi\in C_b$, we have:
		\begin{align*}
		    \langle \widehat{P}^*_t F,\varphi\rangle=
		    \langle F, \widehat{P}_t \varphi\rangle= \int_\R \E_x\Big[\langle \widetilde{Z}_t,\varphi\rangle \Big] F(x)dx = \E_F\Big[\langle \widetilde{Z}_t,\varphi\rangle \Big]=\langle F,\varphi\rangle,
		\end{align*}by \eqref{eq:MTO} and since $F$ is a stationary distribution for \eqref{eq:prob2}.
			\end{proof}

	\begin{prop}
		\label{prop:law?}
		For any bounded measurable real-valued function $\varphi$ on $\R$, we have
		\[
		\mathbb{E}_{m_{T}F}\left[\varphi(Y_{t})\right]=\int_{\mathbb{R}}\varphi(x)m_{T-t}(x)F(x)\ dx,
		\]
In other words, the law of the process $Y_t$ at time $t$, when the initial condition is $m_{T}F$, is given by $m_{T-t}(x)F(x)\ dx$.
	\end{prop}
	\begin{proof}
		We have that, using Equation \eqref{eq:spineSG},
		\[
		\mathbb{E}_{m_{T}F}\left[\varphi(Y_{t})\right]=\Big\langle m_{T}F,\ \frac{1}{m_{T}}\widehat{P}_{t}( \varphi m_{T-t})\Big\rangle=\langle F,\ \widehat{P}_{t}( \varphi m_{T-t})\rangle.
		\]
		Now, Lemma \ref{lem:duality} entails that
		\[
		\langle F,\ \widehat{P}_{t}( \varphi m_{T-t})\rangle=\langle \widehat{P}^{\ast}_{t} F,\  m_{T-t}\varphi\rangle,
		\]
		but $\widehat{P}^{\ast}_{t}F=F$. This gives the result.
	\end{proof}

	\bigskip
	
	It remains now to identify the law of the time-reversal of the process $Y$ started from $m_T F$. To do this, we use the following lemma obtained by applying \cite[Theorem 47]{DM2} in our very specific setting. The proof is given in Appendix \ref{app:reversal}.

	\begin{lem}\label{lem:timereversal}  Let $R=(R_{t})_{t}$ be a positive semigroup which is in duality with a positive semigroup $R^{\ast}$ with respect to the Lebesgue measure.
	Let $f:\mathbb{R}_{+}\times\mathbb{R}\to \mathbb{R}_{+}$ such that $R_{t}f_{t+s}=f_{s}$ for any non-negative real numbers $s$ and $t$. Let $(V_{t})_{t\in\mathbb{R}_{+}}$ be a Markov process with semigroup given by $$\E(\varphi(V_{t+s}\mid\mathcal{F}_{s}   )=\cfrac{1}{f_{s}(V_{s})}R_{t}(f_{s+t}\varphi)(V_{s})$$ and a given initial distribution $\mu$ such that, for any $t\geq 0$,
	the law of $V_{t}$ is absolutely continuous with respect to the Lebesgue measure with density $F\times f_{t}$. Then, the time reversed process at a time $T$ of $X$ is time-homogeneous and has semigroup given by
		\[
		\varphi\to\frac{R_{t}^{\ast}(F\varphi)}{F}.
		\]
	\end{lem}
	
	We will  apply this result stated with general notations, with $f_t(x)=m_{T-t}(x)$, $R_t=\widehat{P}_{t}$ and $Y$ in place of $V$.\\
	
	Now, Propositions \ref{prop:SGmu} and \ref{prop:law?} together with Lemma \ref{lem:duality} allow to apply Lemma \ref{lem:timereversal} in our particular situation. This leads to
	\begin{corollary}
		The time-reversal of the  process $Y$ with initial condition $m_{T}F$ is a Markov process $Y^R$ whose semigroup $P^{R}$ acting on bounded measurable functions is given by
		\[
		P^{R}_{t}\varphi=\cfrac{\widehat{P}^{\ast}_{t}(\varphi F)}{F},\quad \forall \varphi\in B_{b}(\mathbb{R}).
		\]
	\end{corollary}
	
	\medskip
	\begin{rem}\label{rem:killing}
		Let us point out that as $h^{\lambda}$ may be positive, it is always possible that
	$\ \exp\left(\int_{0}^{t}h^{\lambda}(X_{s})\ ds\right)\geq 1.$
		However, taking any positive constant $C\geq c-\lambda$, we can consider the process $X^{C}$ defined by
		\[
		X^{C}_{t}=\begin{cases}
			X_{t}\text{ if }\xi>t\\
			\partial \text{ if }\xi \leq t,
		\end{cases}
		\]
		where $\partial$ is a dummy cemetery state and $\xi$ is a killing time characterized by
		\[
		\mathbb{P}\left(\xi>t\mid \mathcal{F}_{t}\right)=\exp\left(\int_{0}^{t}\left(h^{\lambda}(X_{s})-C\right)\right).
		\]
		Thus, it is easily checked that
		$\ \mathbb{P}\left(\xi>T\right)=e^{-CT}m_{T}(x)\ $
		and
		\[
		\mathbb{E}_{x}\left[\varphi( {X}^{C}_{t})\mathds{1}_{\xi>T}\right]=e^{-CT}\mathbb{E}_{x}\left[\varphi(X_{t})\exp\left(\int_{0}^{t}h^{\lambda}(X_{s})\ ds\right)m_{T-t}(X_{t})\right]=e^{-CT}\widehat{P}_{t}(m_{T-t}\varphi)(x).
		\]
		In particular, we have, for $t\leq T$
		\begin{align*}
			\mathbb{E}_x\left[\varphi({X}^{C}_{t})\mid \xi>T\right]&=\frac{\mathbb{E}\left[\varphi(X_{t})\mathds{1}_{\xi>t}\right]}{\mathbb{P}(\xi>T)}
			=\cfrac{e^{-CT}\widehat{P}_{t}(m_{T-t}\varphi)(x)}{e^{-CT}m_{T}(x)}=\cfrac{\widehat{P}_{t}(m_{T-t}\varphi)(x)}{m_{T}(x)}.
		\end{align*}
		Hence, the law of $X^{C}$ conditioned to $\xi>T$ is the law of the spinal process and is independent of $C$.
		In that sense, $\mu_{x}^{T}$ can be interpreted as the law of $X$ under $\mathbb{P}$ killed at rate $h^{\lambda}$ and conditioned not to be killed.
		In addition, since $\widehat{P}^{\ast}_{t}F=F$, we have, by Lemma \ref{lem:duality}, that
		\[
		\mathbb{E}_F\left[\varphi(X^{C}_{T})\mid \xi>T\right]=\cfrac{\langle F,\ \widehat{P}_{T}\varphi \rangle}{\langle F,\ \widehat{P}_{t}1\rangle}=\langle F,\ \varphi\rangle.
		\]
		Thus, $F$ is a quasi-stationary distribution for the killed process.
	\end{rem}
	
	\section{Return to the initial population process}\label{sec:return}
	
		Let $T>0$ be the time at which we consider the population state $Z^K_T$.  We want to characterize
		the lineage of an individual chosen uniformly in this population of individuals alive at time $T$. Then we associate
		to a past time $s<T$, the trait of the most recent ancestor at this time of a uniformly sampled individual at time $T$. Formally, if $U^{K}_{T}$ is a uniform random variable on $V^{K}_{T}$,
		conditionally on $H^{K}_{T}$, the spinal process $Y^{K}$ is defined by
	\[
	Y^{K}_{t}=X^{U^{K}_{T}}_{t},\quad \forall t\in[0,T],
	\]using the notation in \eqref{eq:EDS_intro}.
	In particular, the law of $Y^{K}$ can be characterized with $H^{K}$ using
	\begin{equation}
		\label{eq:lawY}
		\mathbb{E}_{x}\left[\Phi\left(Y^{K}_{t},\ t\in[0,T]\right)\right]=
		\mathbb{E}_{\delta_{x}}\left[\cfrac{\langle H^{K}_{T},\Phi\rangle}{\langle H^{K}_{T},1\rangle}\right]
	\end{equation}
	with $\Phi:\mathbb{D}([0,T],\mathbb{R})\to\mathbb{R}$  continuous and bounded.\\
	
	\begin{prop}\label{cor:YK-muTx}Assume that $(Z^{K}_0(dx))_{K}$ converges in probability (and weakly as measures) to the deterministic finite measure $\,F(x)dx$. Let $T>0$ be given.
	Then,
	    \begin{equation}
\lim_{K\rightarrow +\infty} \E_{Z^K_0}\left[\Phi\big(Y^K_s,\ s\leq T\big)\right]= \int_\R \langle \mu^T_{x} , \Phi\rangle \ \frac{m_T(x)F(x)}{\lambda} \ dx.\label{eq:hist_tilde}\end{equation}
	 Thus the typical lineage $Y^{K}$ is asymptotically distributed as $Y$ started from a biased initial distribution $\lambda^{-1} m_T(x)F(x)\ dx$.
	\end{prop}
	\begin{proof}From \eqref{eq:lawY}, we have:
	    \begin{align}
\lim_{K\rightarrow +\infty} \E_{Z^K_0}\left[\Phi\big(Y^K_s,\ s\leq T\big)\right]= & 	\E_{Z^K_0} \left[\frac{\langle {H}^K_T,\Phi\rangle}{\langle {H}^K_T,1\rangle}\right]
 =  \lim_{K\rightarrow +\infty} \E_{ Z^K_0}\left[\frac{\langle \widetilde{H}^K_T,\Phi\rangle}{\langle \widetilde{H}^K_T,1\rangle}\right]
\nonumber			\\
= & \frac{1}{\lambda} \int_\R  \Big(m_T(x)\langle \mu_x^T ,\Phi\rangle \Big)\  F(x)\ dx\nonumber\\
			=&  \big\langle \mu^T_{\lambda^{-1} m_T F} , \Phi\big\rangle \label{eq:interp-mu}.\end{align}
			The second equality can be obtained by using Proposition \ref{prop:coupligHHtilde} and following the lines of the proof of Proposition 4.6 in \cite{CHMT}. The third equality is a consequence of Theorem \ref{thm-CV} and \eqref{def:muTx}.
			This gives the announced result.
	\end{proof}

The   asymptotic behavior  of the time reversal of the spinal process $Y^K$  (when $K$ tends to infinity)  is obtained using  Lemma \ref{lem:timereversal}. That is summarized in the following theorem.
We refer to Appendix C \eqref{def:Rcal} for a precise definition  of the map $\mathcal{R}$ that returns time.

		\begin{thm}
		\label{thm:conclusion}
		Under the Assumptions (H) and \eqref{hyp:moment}, and if the sequence $(Z^K_0)_{K}$ converges in probability and weakly to the deterministic measure $F(x)dx$, then we have, for any bounded measurable functions $\Phi$,
		\begin{equation}\label{eq:th_concl}
		\lim\limits_{K\to \infty}\E_{ Z^K_0}\left[\frac{\langle H^K_T,\Phi\circ \mathcal{R}\rangle}{\langle H^K_T,1\rangle}\right]=\mathbb{E}_{F}\left[\Phi\left({Y}^R_s, \,s\in[0,T]\right)\right]
		\end{equation}where ${Y}^R$ is a Markov process with semigroup given by \begin{equation}\label{def:PR}
		    P^R_t\varphi=\cfrac{\widehat{P}_t^{\ast}\big(F\varphi\big)}{F}.
		\end{equation}
	\end{thm}
	\begin{rem}
	Theorem \ref{thm:conclusion} says that, considering $U_{K}$ to be a random variable whose conditional distribution with respect to $H^{K}_{T}$ is uniform on $V^{K}_{T}$, then the process $(X^{U_{K}}_{(T-s)-},\ s\in~[0,T])$ converges in $\mathbb{D}([0,T],\mathbb{R})$ to ${Y}^R$ started from the initial distribution $F$. \\
	For a continuous bounded function $\Phi$, the limit in the left hand side of \eqref{eq:th_concl} equals
	\[\lim\limits_{K\to \infty}\E_{ Z^K_0}\left[\frac{\langle H^K_T,\Phi\circ \mathcal{R}\rangle}{\langle H^K_T,1\rangle}\right]=\frac{1}{\lambda}\E_{F}\left[\langle \Xi_T,\Phi\circ \mathcal{R}\rangle\right].\]Thus, the theorem tells us that the distribution of a typical ancestral lineage (backward in time) in the historical measure $\Xi_T$ has distribution $Y^R$.
	\end{rem}

	We can conclude with the computation of the generator $L^R$ of the process $Y^R$.
	\begin{prop}
	\label{prop:generatorReversed}
The infinitesimal generator $(L^R D(L^R))$  of the Markov process ${Y}^R$ is such that $C^1_b \subset D(L^R)$ and for $\varphi\in C^1_b$, we have
	\begin{equation}\label{eq:LR_conclu}
	L^R\varphi(x)=\rho \varphi'(x) + \gamma\int_{\R} \big(\varphi(y)-\varphi(x)\big)\frac{F(y)}{F(x)}m(y,x)\ dy.
	\end{equation}
	\end{prop}
	
	\begin{proof}
	The infinitesimal generator of the Markov process ${Y}^R$ associated with the semi-group $P^R$ is formally given by
	\[
	L^R\varphi=\cfrac{L^\ast(F\varphi)}{F}+h^{\lambda}\varphi.
	\]Using \eqref{def:Lstar}, we obtain for $\varphi \in C^1_b$,that
	\begin{align*}
	    L^R\varphi(x)= & \rho \varphi'(x) + \gamma \int_{\R} \big(\frac{F(y)}{F(x)}\varphi(y)-\varphi(x)\big)m(y,x)\ dy+\big(\rho \frac{F'(x)}{F(x)} +h(x)-\lambda\big)\varphi(x)\\
	    = & \rho \varphi'(x) + \gamma \int_{\R} \big(\varphi(y)-\varphi(x)\big)\frac{F(y)}{F(x)}m(y,x)\ dy\\
\gamma  &  +\frac{1}{F(x)}\Big(\gamma\int_\R (F(y)-F(x))m(y,x)\ dx+\rho F'(x) +h(x)F(x)-\lambda F(x)\Big)\varphi(x).
	\end{align*}Recall the definition of the operator $\mathcal{A}$ in \eqref{def:Acal}.
	The parenthesis in the last term of the right hand side equals to $L^\ast F+hF-\lambda F=\mathcal{A}F-\lambda F=0$ since $F$ is an eigenvector of $\mathcal{A}$ for the eigenvalue $\lambda$. This provides the announced \eqref{eq:LR_conclu}. The backward ancestral lineage of a typical individual follows the drift $\rho$ and has biased jumps with the jump kernel $F(y)m(y,x)/F(x)$.
	\end{proof}
	
{\footnotesize

\providecommand{\noopsort}[1]{}\providecommand{\noopsort}[1]{}\providecommand{\noopsort}[1]{}\providecommand{\noopsort}[1]{}

}

\appendix

\section{Stochastic differential equations for $Z^K$, $H^K$ and their couplings}\label{app:SDE-ZK-HK}

Recall the birth and death rates $b(x)$ and $d(x)+N^K_t/K$ explained in the Introduction. Let us consider a Poisson point process $N(ds,di,d\theta)$ on $\R_+\times \mathcal{I}\times \R_+$ with intensity measure $ds\otimes n(di)\otimes d\theta$ where $ds$ and $d\theta$ are Lebesgue measures on $\R_+$ and where $n(di)$ is the counting measure on $\mathcal{I}$. Using the Poisson point processes $(Q^i(ds,dy,d\theta), i\in \mathcal{I})$ on $\R_+\times \R\times \R_+$ and with intensity measures the Lebesgue measures that have been defined in the Introduction (see \eqref{eq:EDS_intro}), we can write SDEs satisfied by $Z^K$ and $H^K$ defined in \eqref{def:ZK} and \eqref{def:ZK}.

Let us consider a test function $\varphi\in \Co^1_b(\R_+\times \R,\R)$, then:
\begin{align}
  \lefteqn{ \langle Z^K_t,\varphi(t,.)\rangle=  \int_\R \varphi(t,x)Z^K_t(dx)=\frac{1}{K}\sum_{i\in V^K_t} \varphi\big(t,X^i_t\big)}\nonumber\\
    = & \langle Z^K_0,\varphi(0,.)\rangle + \int_0^t \langle Z^K_s,\partial_s \varphi(s,.)-\rho \partial_x \varphi(s,.)\rangle \ ds\nonumber \\
     + & \int_0^t \int_{\mathcal{I}}\int_{\R_+} \ind_{i\in V^K_{s_-}} \frac{\varphi(s,X^i_{s_-})}{K}\Big( \ind_{\theta\leq b(X^i_{s_-})}  - \ind_{b(X^i_{s_-})<\theta\leq b(X^i_{s_-})+d(X^i_{s_-})+\frac{N^K_t}{K}}\Big) N(ds,di,d\theta)\nonumber\\
     + & \sum_{i\in \mathcal{I}} \int_0^t \int_\R \int_{\R_+}
    \frac{1}{K}\ind_{i\in V^K_{s_-},\theta\leq \gamma m(X^i_{s_-},y)}  \Big(\varphi(s,y)-\varphi(s,X^i_{s_-})\Big)   Q^i(ds,dy,d\theta).\label{eq:eds_Poisson}
\end{align}Using standard Itô calculus (see \cite{ikedawatanabe} and \cite{fourniermeleard}), we obtain that:
\begin{align}
    \langle Z^K_t,\varphi(t,.)\rangle= &\langle Z^K_0,\varphi(0,.)\rangle + \int_0^t \langle Z^K_s,\partial_s\varphi(s,.)-\rho \partial_x\varphi(s,.)+ (h -\langle Z^K_s,1\rangle \varphi(s,.)\rangle \ ds\nonumber\\
    + & \int_0^t\int_{\R} \gamma \int_{\R} \big(\varphi(s,y)-\varphi(s,x)\big) m(x,y)dy\ Z^K_s(dx)\ ds+M^{K,\varphi}_t
 \end{align}where $M^{K,\varphi}$ is a squre integrable martingale with predictable quadratic variation process:
 \begin{align}
     \langle M^{K,\varphi}\rangle_t= & \frac{1}{K} \int_0^t \int_{\R} \big(b(x)+d(x)+\langle Z^K_s,1\rangle\big)\varphi^2(s,x)\  Z^k_s(dx)\ ds\nonumber\\
     + &
     \frac{1}{K} \int_0^t \int_\R  \int_\R \gamma \big(\varphi(s,y)-\varphi(s,x)\big)^2 m(x,y)dy\ Z^K_s(dx)\ ds.
 \end{align}

We proceed similarly for $H^K$. For finite measures on $\mathbb{D}$, we will consider test functions $\Phi_\varphi$ defined for $\Phi\in \Co^1_b(\R,\R)$ and $\varphi\in \Co^1(\R_+\times \R,\R)$ by
	$\
\Phi_\varphi(y)=\Phi\Big(\int_0^T \varphi(t,y_t) \ dt\Big)\ $
 for $y\in \D$. Recall that the paths $y$ in the support of $H^K_t$ for $t\in [0,T]$ are constant after time $t$.

For such test function, time $t\in [0,T]$ and path $y\in \supp(H^K_t)\subset \D$, let us define the derivative
\begin{equation*}
\mathcal{D}\Phi_\varphi(t,y)=\Phi'\Big(\int_0^T \varphi(s,y_s)\ ds\Big) \int_t^T \partial_x \varphi(s,y_t)\ ds.
\end{equation*}Also, for $y\in \D$, $s\in [0,T]$ and $x\in \R$, we will denote by $(y|s|x)$ the càdlàg path defined as
\[(y|s|x)(t)=\begin{cases}
y(t) & \mbox{ if }t<s,\\
x & \mbox{ if }t\geq s.
\end{cases}\]

For such test function $\Phi_\varphi$,
\begin{align}
   \lefteqn{\langle H^K_t,\Phi_\varphi\rangle=    \langle H^K_0,\Phi_\varphi\rangle + \int_0^t \langle H^K_s,-\rho \mathcal{D}\Phi_\varphi(s,.)\rangle \ ds\nonumber }\\
     + & \int_0^t \int_{\mathcal{I}}\int_{\R_+} \ind_{i\in V^K_{s_-}} \frac{\Phi_\varphi(X^i)}{K}\Big( \ind_{\theta\leq b(X^i_{s_-})}  - \ind_{b(X^i_{s_-})<\theta\leq b(X^i_{s_-})+d(X^i_{s_-})+\frac{N^K_t}{K}}\Big) N(ds,di,d\theta)\nonumber\\
     + & \sum_{i\in \mathcal{I}} \int_0^t \int_\R \int_{\R_+}
    \frac{1}{K}\ind_{i\in V^K_{s_-},\theta\leq \gamma m(X^i_{s_-},y)}  \Big(\Phi_\varphi(X^i|s|y)-\Phi_\varphi(X^i)\Big)   Q^i(ds,dy,d\theta)\label{eq:eds_Poisson-H}\\
    = &  \langle H^K_0,\Phi_\varphi\rangle + \int_0^t \langle H^K_s,-\rho \mathcal{D}\Phi_\varphi(s,.)\rangle \ ds\nonumber \\
     + & \int_0^t \int_{\D} \Big(  h(y_{s}) -\langle H^K_s,1\rangle\Big)\Phi_\varphi(y)\  H^K_s(dy)\ ds\nonumber\\
     + &  \int_0^t \int_\D \gamma \int_{\R}
      \Big(\Phi_\varphi(y|s|x)-\Phi_\varphi(y)\Big)  m(y_s,x) \ dx\ H^K_s(dy)\ ds+ M^{K,\Phi,\varphi}_t,
\end{align}where $M^{K,\Phi,\varphi}$ is a square integrable martingale with predictable quadratic variation process:

\begin{align*}
    \langle M^{K,\Phi,\varphi}\rangle_t= & \frac{1}{K} \int_0^t \int_{\D} \big(b(y_s)+d(y_s)+\langle H^K_s,1\rangle\big)\Phi_\varphi^2(x)\  H^k_s(dy)\ ds\nonumber\\
     + &
     \frac{1}{K} \int_0^t \int_\D  \int_\R \gamma \big(\Phi_\varphi(y|s|x)-\Phi_\varphi(y)\big)^2 m(y_s,x)dx\ H^K_s(dy)\ ds.
\end{align*}

The processes $\widetilde{Z}^K$ and $\widetilde{H}^K$ are constructed similarly to \eqref{eq:eds_Poisson} and \eqref{eq:eds_Poisson-H} with $N^K_t/K$ replaced by $\lambda$, and with the same initial conditions, Poisson point processes and motion processes. Therefore they are solutions of the following equations.

For $\varphi\in D(L)$ and $\Phi\in \Co^1_b(\R,\R),$
	\begin{equation}
		\label{ref:partSys}
		\langle\widetilde{Z}^{K}_{t},\varphi\rangle=\langle {Z}^{K}_{0},\varphi\rangle+ \int_0^t \int_{\R} \left\{ h^{\lambda}(x)\,\varphi(x)+L\varphi(x)\right\}\widetilde{Z}^{K}_{s}(dx) \ ds+\widetilde M^{K,\varphi}_{t},
	\end{equation}
	where $\widetilde M^{K,\varphi}$ is a square integrable martingale, and for the historical process,
	\begin{multline}
	    \langle \widetilde{H}^K_t,\Phi_\varphi\rangle
	        =   \langle H^K_0,\Phi_\varphi\rangle + \int_0^t \int_\D -\rho \mathcal{D}\Phi_\varphi(s,y) \\
	        +
    \Big( h^\lambda (y_s) +  \gamma \int_{\R}
      \big(\Phi_\varphi(y|s|x)-\Phi_\varphi(y)\big)  m(y_s,x) \ dx\Big)\Phi_\varphi(y)\  \widetilde{H}^K_s(dy)\ ds+ \widetilde{M}^{K,\Phi,\varphi}_t
	\end{multline}
	where $\widetilde{M}^{K,\Phi,\varphi}$ is a square integrable martingale.

\section{Time Reversal of Markov processes}\label{app:reversal}

In this section, we consider the problem of reversing time for Markov processes. Let $T>0$ be fixed and let us consider the
linear map $\mathcal{R}:\mathbb{D}([0,T],\mathbb{R})\to\mathbb{D}([0,T],\mathbb{R})$ defined by
		\begin{equation}\label{def:Rcal}
		\mathcal{R}(\varphi)(s)=\begin{cases}
			\lim\limits_{\ve\to 0,\ \ve>0}\varphi(T-s-\ve) &\text{ if } s\neq T\\
			\varphi(0)&\text{ if } s=T
		\end{cases}
		\end{equation}

	Let us start with the following lemma.
	\begin{lem}
		The linear map $\mathcal{R}$ is $1$-Lipschitz continuous for the Skorokhod topology.
	\end{lem}
	\begin{proof}
		Let $\varphi,\Psi\in \mathbb{D}([0,T],\mathbb{R})$, and $\lambda:[0,T]\to[0,T]$ which is increasing continuous and satisfying $\lambda(0)=0$ and $\lambda(T)=T$.
		Then, we have
		\begin{multline*}
			\sup_{s\in[0,T]}\left|\mathcal{R}\left(\varphi-\Psi\right)(\lambda(s))\right|\\=\max\left\{\sup_{s\in(0,T]}\left|\lim\limits_{\ve\to0,\ \ve>0}\left(\varphi(T-\lambda(s)-\ve)-\Psi(T-\lambda(s)-\ve)\right)\right|,\ |\varphi(0)-\psi(0)|\right\}.
		\end{multline*}
		Now, choosing $s\in(0,T]$ and $\ve>0$, we define
		\[
		\theta(s)=\inf\{t>s\mid \lambda(t)\geq \lambda(s) +\ve\}
		\]
		which is defined for all $s\in\lambda^{-1}([0,T-\ve])$ with values in $[\theta(0),T]$. Thus, for all $s\in\lambda^{-1}([0,T-\ve])$, we have
		\begin{multline*}
			\left|\left(\varphi(T-\lambda(s)-\ve)-\Psi(T-\lambda(s)-\ve)\right)\right|=\left|\left(\varphi(T-\lambda(\theta(s)))-\Psi(T-\lambda(\theta(s)))\right)\right|\\ \leq \sup_{s\in[0,T]}\left|\left(\varphi(T-\lambda(s))-\Psi(T-\lambda(s))\right)\right|
		\end{multline*}
		leading to
		\[
		\sup_{s\in[0,T]}\left|\mathcal{R}\left(\varphi-\Psi\right)(\lambda(s))\right|\leq\sup_{s\in[0,T]}\left|\left(\varphi(T-\lambda(s))-\Psi(T-\lambda(s))\right)\right|.
		\]
		The result easily follows.
	\end{proof}

The proof of the following proposition is based on \cite[Theorem 47]{DM2}.

	\begin{lem}\label{lem:timereversal-app}  Let $R=(R_{t})_{t}$ be a positive semigroup which is in duality with a positive semigroup $R^{\ast}$ with respect to the Lebesgue measure. Let $f:\mathbb{R}_{+}\times\mathbb{R}\to \mathbb{R}_{+}$ such that $R_{t}f_{t+s}=f_{s}$ for any non-negative real numbers $s$ and $t$. Let $(V_{t})_{t\in\mathbb{R}_{+}}$ be a Markov process with semigroup given by $$\E(\varphi(V_{t+s}\mid\mathcal{F}_{s}   )=\cfrac{1}{f_{s}(V_{s})}R_{t}(f_{s+t}\varphi)(V_{s})$$ and a given initial distribution $\mu$ such that, for any $t\geq 0$, the law of $V_{t}$ is absolutely continuous with respect to the Lebesgue measure with density $F\times f_{t}$. Then, the time reversed process $\mathcal{R}(X)$ at a time $T$ of $X$ is time-homogeneous and has semigroup given by
		\[
		\varphi\to\frac{R_{t}^{\ast}(F\varphi)}{F}.
		\]
	\end{lem}

	\begin{proof}
		Let $(X_{t})_{t\in \mathbb{R}_{+}}$ be a time-inhomogeneous Markov process with evolution family $(Q_{s,t})_{0\leq s\leq t}$ which characterized the dynamics of $X$ in the sense that
		\[
		\mathbb{E}\left[\varphi(X_{s+t})\mid X_{s}\right]=Q_{s,s+t}\varphi(X_{s}).
		\]
		We consider the space-time-ification of $Q$ as a semigroup $\tilde P$ defined for all bounded measurable function $\varphi:\mathbb{R}_{+}\times \mathbb{R}\to \mathbb{R}_{+}$ by
		\[
		\widetilde{P}_{t}\varphi(s,x)=(Q_{s,t+s}\varphi(t+s,\cdot))(x)=\cfrac{1}{f_{s}}P_{t}(f_{t+s}\varphi_{t+s}).
		\]
		The semigroup
		\[
		\widetilde{P}^{\ast}_{t}\varphi(s,x)=\mathds{1}_{s>t}(s,x)(P^{\ast}_{t}\varphi(s-t,\cdot))(x)
		\]
		is in duality with $P$ with respect to the measure $f_{s}(x)\lambda(ds)\ dx$. Indeed, for any bounded measurable functions $\varphi$ and $g$ on $\mathbb{R}_{+}\times\mathbb{R}$, we have
		\begin{align*}
			\int_{\mathbb{R}_{+}\times  \mathbb{R}}g(x,s)P_{t}(f_{t+s}\varphi_{t+s})(x)\ ds\ dx&=\int_{\mathbb{R}_{+}}\int_{\mathbb{R}}f_{t+s}(x)\varphi_{t+s}(x)P_{t}^{\ast}g(s,x)\ dx\ ds\\
			&=\int_{[t,\infty)}\int_{\mathbb{R}}f_{s}(x)\varphi(s,x)P^{\ast}_{t}g(s-t,x)\ dx\ ds\\
			&=\int_{\mathbb{R}_{+}}\int_{\mathbb{R}}\varphi(s,x)\widetilde{P}^{\ast}_{t}g(s,x)\ f_{s}(x)dx\ ds.
		\end{align*}
		The semigroup $\widetilde P$ induces a Markov process on $\mathbb{R}_{+}\times \mathbb{R}$ for which we assume the initial condition to be $\mu'=\delta_{0}\otimes \mu$.
		The hypothesis of \cite[Chapter XVIII.46]{DM2} is that the potential measure $\mu' U'$ defined, for a measurable $\mathcal{A}\subset\mathbb{R}_{+}\times \mathbb{R}$, by
		\[
		\mu' U'(\mathcal{A})=\int_{\mathbb{R}_{+}\times \mathbb{R}}\int_{0}^{\infty}\widetilde{P}_{t}\mathds{1}_{\mathcal{A}}(e)\ dt \ \mu'(de)
		\]
		has a density $k:\mathbb{R}_{+}\times \mathbb{R}\to \mathbb{R}_{+}$ with respect to $f_s(x)ds dx$. To see this, let $A$ and $B$ two measurable subset of respectively $\mathbb{R}_{+}$ and $\mathbb{R}$. First, we have
		\[
		\widetilde{P}_{t}\mathds{1}_{A\times B}(s,x)=\mathds{1}_{A}(s+t)Q_{s,t+s}\mathds{1}_{B}(x).
		\]
		Thus,
		\begin{align*}
			\mu' U'(A\times B)&=\int_{\mathbb{R}_{+}\times \mathbb{R}}\int_{0}^{\infty}\mathds{1}_{A}(s+t)Q_{s,t+s}\mathds{1}_{B}(x)\ dt \ \ \delta_{0}(ds)\ \mu(dx)\\
			&=\int_{\mathbb{R}_{+}\times \mathbb{R}}\int_{s}^{\infty}\mathds{1}_{A}(t)Q_{s,t}\mathds{1}_{B}(x)\ dt \ \ \delta_{0}(ds)\ \mu(dx)\\
			&=\int_{ \mathbb{R}}\int_{0}^{\infty}\mathds{1}_{A}(t)Q_{0,t}\mathds{1}_{B}(x)\ dt \ \mu(dx)=\int_{A\times B} f_{t}(x)F(x)\ dt\ dx.
		\end{align*}
		Thus, the density of $\mu'U'$ with respect to $f_{s}(x)\ ds\ dx$ is given by $F$ and is time independent.
		Then, the time reversed process $\widetilde{X}^{R}$ of $\widetilde{X}$
		is given by
		\[
		\widetilde{X}^{R}_{t}=\widetilde{X}_{(T-t)-}=(T-t,X_{(T-t)-})
		\]and has semigroup given by
		\[
		\widetilde{P}'_{t}\varphi(s,x)=\cfrac{\widetilde{P}^{\ast}_{t}(F\varphi)(s,x)}{F}.
		\]
		Now, let $\phi:\mathbb{R}\mapsto\mathbb{R}$ be a bounded measurable maps and set $\varphi(s,x)=\phi(x)$. Thus, for $t+s<T$,
		\begin{align*}
			\mathbb{E}\left[\phi(X_{(T-(t+s))-})\bigg| X_{(T-s)-}\right]&=\mathbb{E}\left[\varphi(\widetilde{X}'_{t+s})\bigg| \widetilde{X}'_{s}\right]=\widetilde{P}'_{t}\varphi(\widetilde{X}'_{s})=\widetilde{P}'_{t}\varphi(T-s,X_{(T-s)-})
			\\&=\cfrac{P^{\ast}_{t}(F \varphi)(X_{s})}{F(X_{s})}.
		\end{align*}
	\end{proof}

\section{Absolute continuity of the solution of \eqref{limit-moving}}\label{app:density}

Let us prove Proposition \ref{prop:density}.\\

Recall the idea of the proof. If \eqref{eq:pde} possesses a solution $f_{t}$ in $C([0,T],\mathbb{L}^{1})$, then $f_{t}(x)dx$ is solution of \eqref{limit-moving},
and the identification $\xi_t=f_{t}(x)dx$ follows from the uniqueness of the solution of \eqref{limit-moving}. Thus, we only have to prove that  $f_{0}\in \mathbb{L}^1$
yields a solution  $f_{t}$ in $C([0,T],\mathbb{L}^{1})$. To prove this, we follow closely the computation in \cite{dyson2000nonlinear}. \\

Consider the semigroup $S_{t}$ acting on $\mathbb{L}^{1}$ with generator $L^{\ast}+h$.
The set
\[
W=\left\{F\in C([0,T],\mathbb{R}) \mid F\geq 0 \text{ and }\sup_{0\leq t\leq T}|F(t)|\leq \sup_{0\leq t\leq T}\|S_{t}f_{0}\|_{1} \right\}.
\]
is convex, closed and bounded. Now, consider the operator $K$ acting on $C([0,T],\mathbb{R})$ and defined by for all $F\in C([0,T],\mathbb{R})$ by
\[
KF(t)=\exp\left(-\int_{0}^{t}F(u)\ du\right)\|S_{t}f_{0}\|_1.
\]
Let us prove that $K$ is a compact operator. Take $G\in K(B(0,1))$ so that $G=KF$ for some $F\in B(0,1)$ where $B(0,1)$ is the open unit ball of $C([0,T],\mathbb{R})$. So, we have, with $t>s$,
\begin{multline*}
|G(t)-G(s)| \\
\begin{aligned}
&\leq e^{t}\|S_{t}f_{0}-S_{s}f_{0}\|_{1}+e^{t}\|S_{s}f_{0}\|_{1}\left|\exp\left(-\int_{0}^{t}(F(u)+1)\ du\right)-\exp\left(-\int_{0}^{s}(F(u)+1)\ du\right)\right|\\
&\leq e^{T}\|S_{t}f_{0}-S_{s}f_{0}\|_{1}+\|S_{s}f_{0}\|_{1}e^{T}\left|\exp\left(-\int_{s}^{t}(F(u)+1)\ du\right)-1\right|\\&\leq e^{T}\|S_{t}f_{0}-S_{s}f_{0}\|_{1}+\|S_{s}f_{0}\|_{1}e^{T}\left|\int_{s}^{t}(F(u)+1)\ du\right|\\&\leq e^{T}\|S_{t}f_{0}-S_{s}f_{0}\|_{1}+2\|S_{s}f_{0}\|_{1}e^{T}(t-s).
\end{aligned}
\end{multline*}
So the family $K(B(0,1))$ is equibounded and equicontinuous, and Arzelà–Ascoli theorem entails the compactness of $K$. It is easy to check that $K(W)\subset W$ and thus Leray-Schauder fixed point theorem gives the existence
of a fixed point $F^{\ast}$ for $K$ in $W$. \\
The function $f$ defined for $x\in \R$ by
\[
f_{t}(x)=\exp\left(-\int_{0}^{t}F^{\ast}(u)\ du\right)S_{t}f_{0}(x),
\]
provides the desired solution and ends the proof of Proposition \ref{prop:density}.\hfill $\Box$

\end{document}